\documentclass[11pt]{amsart}

\usepackage{amssymb,amsmath,amscd}
\usepackage{mathtools}
\usepackage{bm}
\usepackage{amsfonts}
\usepackage{graphicx}
\usepackage{xcolor}
\usepackage{mathrsfs}
\usepackage{stmaryrd}
\usepackage{epsfig,color}
\usepackage{blindtext}
\usepackage{enumerate}
\usepackage{enumitem}
\usepackage{geometry}
\usepackage{color}
\usepackage[lite, alphabetic]{amsrefs}

\usepackage[colorlinks,linktocpage]{hyperref}
\hypersetup{linkcolor=[rgb]{0,0,0.715}}
\hypersetup{citecolor=[rgb]{0,0.715,0}}

\newtheorem{thm}{Theorem}[section]
\newtheorem{cor}[thm]{Corollary}
\newtheorem{prop}[thm]{Proposition}
\newtheorem{lem}[thm]{Lemma}

\newtheorem{quest}[thm]{Question}

\newtheorem{mainthm}{Theorem}

\theoremstyle{definition}
\newtheorem{defn}[thm]{Definition}

\newtheorem{exmp}[thm]{Example}

\newtheorem{notn}[thm]{Notation}

\usepackage{comment}

\theoremstyle{remark}
\newtheorem{rem}[thm]{Remark}

\newcommand{\Ch}{{\sf Ch}}
\let\phi\varphi

\newcommand{\hess}{{\mathrm{Hess}}}
\DeclareMathOperator{\Reg}{\mathcal{R}}
\DeclareMathOperator{\lip}{\mathrm{lip}}
\DeclareMathOperator{\Lip}{\mathrm{Lip}}
\DeclareMathOperator{\Spt}{\mathrm{Spt}}

\newcommand{\C}{\mathbb{C}}
\newcommand{\N}{\mathbb{N}}

\newcommand{\R}{\mathbb{R}}
\newcommand{\Z}{\mathbb{Z}}
\renewcommand{\subset}{\subseteq}
\newcommand{\defeq}{\mathrel{\mathop:}=}

\newcommand{\vol}{\mathrm{vol}}
\newcommand{\dvol}{\mathrm{dvol}}

\newcommand{\Ric}{\mathrm{Ric}}
\newcommand{\dist}{d}

\newcommand{\meas}{\mathfrak{m}}
\newcommand{\m}{\mathfrak{m}}

\newcommand{\di}{\mathop{}\!\mathrm{d}}

\DeclareMathOperator{\RCD}{RCD}
\DeclareMathOperator{\CDe}{CD}
\DeclareMathOperator{\RV}{\mathrm{RV}}

\makeatletter
\let\c@equation\c@thm
\makeatother
\numberwithin{equation}{section}

\title[Ricci curvature and slow relative volume growth]{Complete manifolds with nonnegative Ricci curvature and slow relative volume growth}

\author{Dimitri Navarro}
\address[Dimitri Navarro]{Department of Mathematics, University of California, Santa Cruz, CA, USA.}
\email{dnavar17@ucsc.edu}

\author{Jiayin Pan}
\address[Jiayin Pan]{Department of Mathematics, University of California, Santa Cruz, CA, USA.}
\email{jpan53@ucsc.edu}

\author{Xingyu Zhu}
\address[Xingyu Zhu]{Michigan State University, East Lansing, MI, USA.}
\email{zhuxing3@msu.edu}

\begin{document}

\begin{abstract}
    For any complete and noncompact manifold $M$ with $\Ric\ge 0$, we define a function $\RV(s)$ that describes the growth of relative volume asymptotically
    $$\RV(s)=\limsup_{r\to\infty} \dfrac{\vol B_{rs}(p)}{\vol B_r(p)},\quad s\ge 1.$$
    Then we study the fundamental groups of such manifolds with slow relative volume growth and sublinear diameter growth. We show that if $\RV(s)\ll s^2$ as $s\to\infty$, then $\pi_1(M)$ is almost abelian; if $\RV(s)\ll s^{1+\delta}$ for some $\delta\in (0,1)$ and the Ricci curvature is positive at a point, then $\pi_1(M)$ is finite. These results generalize our previous work on complete manifolds with $\Ric\ge 0$ and linear (minimal) volume growth \cite{NPZ}.
\end{abstract}


\maketitle

\tableofcontents

\parskip=4pt plus 0.5pt

\section{Introduction}\label{sec:intro}
This paper is concerned with the fundamental groups of complete Riemannian manifolds with nonnegative Ricci curvature. For any closed manifold with $\Ric\ge 0$, its fundamental group is virtually abelian by the work of Cheeger--Gromoll \cite{CG_split}. In contrast, the fundamental group of an open (i.e., complete and noncompact) manifold with $\Ric\ge 0$ may not have any abelian subgroups of finite index \cite{Wei88}, although it is always virtually nilpotent \cites{Milnor68,Gromov81,KW}. Geometrically, this difference is reflected in the behavior of covering spaces. For a closed manifold \(M\) with \(\Ric \ge 0\), any torsion-free element in $\pi_1(M)$ gives rise to a line in the universal cover $\widetilde{M}$ of $M$, in turn, $\widetilde{M}$ splits off a line isometrically by Cheeger--Gromoll splitting theorem \cite{CG_split}. This splitting structure on the covering space also implies a classical theorem by Bonnet--Myers: if a closed manifold $M$ has $\Ric>0$, then $\pi_1(M)$ is finite. For an open manifold $M$ with $\Ric\ge 0$, however, a torsion-free element in $\pi_1(M)$ may not produce a line in the universal cover $\widetilde{M}$ and thus $\widetilde{M}$ may not split isometrically. Therefore, identifying classes of open manifolds with $\Ric\ge 0$ whose fundamental groups remain virtually abelian is one of the main problems in understanding the interaction between Ricci curvature and fundamental group. 

In our previous work \cite{NPZ}, we found that the volume growth of $M$ plays a prominent role in this question, and in this paper we will further generalize these results. Recall that $M$ has linear (or minimal) volume growth if
$$\limsup_{r\to\infty} \dfrac{\vol(B_r(p))}{r}<\infty.$$
This class of manifolds has particular interest because open manifolds with $\Ric\ge 0$ have at least linear volume growth by the independent works of Calabi \cite{Calabi75} and Yau \cite{Yau76}. Manifolds with linear volume growth were studied extensively by Sormani in the late 90s \cites{Sor98,Sor00a,Sor00b}. 

Linear volume growth implies the following geometric properties:\\
$\bullet$ $M$ has stable linear volume growth (see \cite{ZZ}*{Remark 2.1} or \cite{NPZ}*{Appendix})
$$\lim_{r\to\infty} \dfrac{\vol(B_r(p))}{r}=C\in(0,\infty).$$
$\bullet$ Either $M$ splits isometrically as $\R\times K$, where $K$ is a closed manifold, or $M$ has sublinear diameter growth (see \cite{Sor00a})
$$\lim_{r\to\infty} \dfrac{\mathrm{diam}(\partial B_r(p))}{r}=0,$$
where the diameter is measured under the extrinsic distance; this property particularly implies that $\pi_1(M)$ is finitely generated \cite{Sor00b}, which is not true for general open manifolds with $\Ric\ge 0$ \cites{BNS_6,BNS_7}.\\
In \cite{NPZ}, we proved the following results on manifolds with linear volume growth.\\
(1) \cite{NPZ}*{Theorem A} If an open manifold $M^n$ has $\Ric\ge0$ and linear volume growth, then $\pi_1(M)$ contains a $\Z^k$ subgroup of finite index, where $0\le k\le n-1$.\\
(2) \cite{NPZ}*{Theorem B} If an open manifold $M^n$ has $\Ric>0$ and linear volume growth, then $\pi_1(M)$ is finite. Recently, Huang--Huang further improved this finiteness result by weakening the $\Ric>0$ condition to $\Ric\ge0$ and $\Ric_p>0$ at a point $p$; in fact, they proved a splitting result on any $\Z$-folding covering space of $M$ \cite{huang2025}.\\
They can be viewed as extensions of the virtual abelianness and finiteness results for closed manifolds mentioned above. In the proofs of these results on linear volume growth, both stable linear volume growth and sublinear diameter growth are essential.

The main goal of this paper is to generalize the main results in \cite{NPZ} by relaxing the linear volume growth condition. For this purpose, we introduce a relative volume growth function as below. Given an open manifold $(M,p)$ of $\Ric\ge 0$, we define its relative volume growth function $\RV:[1,\infty)\to [1,\infty)$ by
$$\RV(s):= \limsup_{r\to\infty} \dfrac{\vol B_{rs}(p)}{\vol B_r(p)}.$$
This function $\RV(s)$ does not depend on the choice of the base point $p$ (Proposition \ref{prop:RV_well_defined}). By Bishop--Gromov relative volume comparison, $\RV(s)\le s^n$ for all $s\ge 1$, where $n=\dim M$. When $M$ has linear volume growth, $\RV(s)=s$ because $M$ has stable linear volume growth. We also remark that for any open manifold $M$, the condition $\RV(s)\ll s^\beta$, where $\beta>0$, implies $\vol B_r(p)\ll r^{\beta}$ (Lemma \ref{lem:rel_vs_abs_vol}), but the converse is not true in general.

Now we state the main results of this paper.

\begin{mainthm}\label{mainthm:vir_abel}
   Let $M$ be an open manifold with $\Ric\ge 0$. If $M$ has sublinear diameter growth and
   $$\lim_{s\to\infty} \dfrac{\RV(s)}{s^2} =0,$$
   then $\pi_1(M)$ contains a $\Z^k$ subgroup of finite index, where $0\le k\le n-1$.
\end{mainthm}

\begin{mainthm}\label{mainthm:finite}
   Let $M$ be an open manifold with $\Ric\ge 0$ and $\Ric_p>0$ at a point $p$. If $M$ has sublinear diameter growth and
   $$\lim_{s\to\infty} \dfrac{\RV(s)}{s^{1+\delta}} =0$$ 
   for some $\delta\in(0,1)$, then $\pi_1(M)$ is finite.
\end{mainthm}

We suspect that Theorem \ref{mainthm:finite} may not hold if one considers the absolute volume growth function instead of the relative one (Cf. \cite{NPZ}*{Question 1.4}). 

\begin{quest}
   Construct an open manifold $M$ with $\Ric>0$, sublinear diameter growth,
   $$\lim_{r\to\infty} \dfrac{\vol B_r(p)}{r^{1+\delta}}=0,$$
   for some $\delta\in (0,1)$,
   but $\pi_1(M)=\Z$.
\end{quest}

\begin{rem}
   We point out that the smallest volume growth among the known examples with $\Ric>0$ and $\pi_1(M)=\Z$ is quadratic (see \cite{NPZ}*{Remark 1.5}); the example has $\vol B_r(p)\sim C r^2$ as $r\to\infty$, thus $\RV(s)=s^2$.
\end{rem}

\begin{quest}
  On an open manifold $M$ with $\Ric\ge 0$, does $\RV(s) \ll s^2$ as $s\to\infty$ imply sublinear diameter growth? 
\end{quest}

Compared to \cite{NPZ}, the new input mainly comes from a sharp plane/halfplane rigidity for $\RCD(0,N)$ spaces, which in turn implies a splitting result on the asymptotic cones of the covering spaces of $M$. 

\begin{thm}\label{thm:rigid_p/h}
  Let $(Y,y,\dist,\mathfrak{m})$ be a pointed $\RCD(0,N)$ space. Suppose that\\
(1) $Y$ has an isomorphic $G$-action, where $G\simeq \mathbb{R}\times K$ is a closed subgroup of $\mathrm{Isom}(Y)$, $K$ is abelian and fixes $y$, and the quotient metric space $(Y/G,\bar{y})$ is isometric to a ray $([0,\infty),0)$;\\
(2) the measure $\meas$ on regions $\Omega_r\subseteq Y$ (see Definition \ref{defn:intro_omega_r} below) satisfies
$$\lim_{r\to 0^+} \dfrac{\meas(\Omega_r)}{r^2}=\infty,\quad \lim_{r\to\infty} \dfrac{\meas(\Omega_r)}{r^2}=0.$$
Then the metric space $(Y,y,\dist)$ is isometric to a Euclidean plane $(\mathbb{R}^2,0)$ or a Euclidean halfplane $(\mathbb{R}\times [0,\infty),0)$.
\end{thm}

\begin{defn}\label{defn:intro_omega_r}
Under the condition (1) of Theorem \ref{thm:rigid_p/h}, the region $\Omega_r\subseteq Y$ is defined by
$$\Omega_r:=\{ (w,h) \cdot \gamma(t)\ |\ w\in[-1,1], h\in K, t\in[0,r] \},$$
where $\gamma$ is a lift of the unit speed ray in $Y/G=[0,\infty)$ at $y$ and $(w,h)$ represents an element in $\mathbb{R}\times K\simeq G$.
\end{defn}

Theorem \ref{thm:rigid_p/h} substantially generalizes \cite{NPZ}*{Theorem 1.8}, where the measure $\meas$ is assumed to be linear along the regions $\Omega_r$. Here in Theorem \ref{thm:rigid_p/h}, we only assumed that $\meas(\Omega_r)$ is super-quadratic at $0$ and sub-quadratic at infinity.

Theorem \ref{thm:rigid_p/h} will follow from two rigidity results, concerning free $\R$-action and isotropic $S^1$-action separately.

\begin{thm}[Sharp halfplane rigidity]\label{thm:intro_halfplane}
   Let $(Y,y,d,\mathfrak{m})$ be an $\RCD(0,N)$. Suppose that\\
   (1) $Y$ has an isomorphic $G$-action, where $G\simeq \mathbb{R}$ is closed in $\mathrm{Isom}(Y)$, such that the quotient metric space $(Y/G,\bar{y})$ is isometric to a ray $([0,\infty),0)$;\\
   (2) the measure $\meas$ satisfies $$\lim_{r\to\infty}\dfrac{\meas (\Omega_r)}{r^2}=0.$$
   Then the metric space $(Y,d)$ is isometric to a Euclidean halfplane.
\end{thm}

\begin{thm}[Sharp ray rigidity]\label{thm:intro_ray}
   Let $(Y,y,d,\mathfrak{m})$ be an $\RCD(0,N)$ space. Suppose that\\
   (1) $Y$ has an isomorphic $S^1$-action, where $S^1$ fixes $y$, such that the quotient metric space $(Y/S^1,\bar{y})$ is isometric to a ray $([0,\infty),0)$;\\
   (2) the measure $\meas$ satisfies 
   $$\lim_{r\to 0^+}\dfrac{\meas (B_r(y))}{r^2}=\infty.$$
   Then the metric space $(Y,d)$ is isometric to a ray and the $S^1$-action on $Y$ is trivial.
\end{thm}

\begin{rem}
In both theorems above, the degree $2$ in the condition (2) is sharp. 
See Example \ref{exmp:hp_2_sharp} and Remark \ref{rem:ray_2_sharp}.
\end{rem}

We point out that the proofs of these new rigidity results are distinct from the linear case. The method in \cite{NPZ}*{Proposition 4.2} relies on the local functional splitting for $\RCD$ spaces (see \cites{KKL23,Ketterer_23} after \cite{Gigli13}), and this method no longer applies here because $\meas(\Omega_r)$ is not linear. To prove these improved rigidity results, we recover a continuous and $W^{1,2}_{loc}$ Riemannian metric on the an open and dense subset of $Y$ and apply a distributional Bakry--Em\'ery Ricci curvature computation. The usage of distributional Ricci curvature is inspired by the work of Mondino--Ryborz \cite{MondinoRyborz}. The readers may refer to Section \ref{subsec:halfplane_statements} for an outline of this approach. 

\begin{rem}
  While preparing this manuscript, we became aware of a recent preprint by Cucinotta, Magnabosco, and Semola \cite{CMS26}, in which a rigidity result for equivariant asymptotic cones is established under a different setting; see \cite{CMS26}*{Theorem 3.5}. Although their rigidity result overlaps with ours, and both proofs rely on distributional lower Ricci curvature bounds from \cite{MondinoRyborz}, our work is carried out independently and does not depend on the results in \cite{CMS26}.
\end{rem}

\noindent\textbf{Organization of the paper.} In Section \ref{sec:pre}, we recall the concepts and key results about $\RCD$ spaces and asymptotic cones of complete manifolds with $\Ric\ge 0$. Section \ref{sec:circle ray} is the technical part of the paper, where we study an $\RCD(0,N)$ space $(Y,d,\meas)$ with isomorphic $S^1$-action whose quotient metric space is isometric to a ray. Section \ref{sec:circle ray} concludes with nonnegative distributional Bakry--Em\'ery Ricci curvature on an open and dense subset of $Y$. The reader may refer to Section \ref{subsec:halfplane_statements} for an outline of this section. In the following Section \ref{sec:rigidity}, we use the distributional Ricci computation to derive the rigidity results Theorems \ref{thm:rigid_p/h}, \ref{thm:intro_halfplane}, and \ref{thm:intro_ray}. Then we use Theorem \ref{thm:rigid_p/h} to prove Theorems \ref{mainthm:vir_abel} and \ref{mainthm:finite} in Section \ref{sec:proof}.

\noindent\emph{Acknowledgments.} J. Pan is partially supported by the National Science Foundation DMS-2304698 and Simons Foundation Travel Support for Mathematicians. D. Navarro and X. Zhu are partially supported by the AMS-Simons Travel Grant.

\section{Preliminaries}\label{sec:pre}
\subsection{RCD spaces}

The purpose of this section is to introduce metric measure spaces satisfying the $\RCD(0,N)$ condition, which will play an essential role in the sequel. Indeed, our approach involves blow-downs (also called asymptotic cones) of complete Riemannian manifolds with nonnegative Ricci curvature. Such blow-downs, equipped with a limit renormalized measure, satisfy the $\RCD(0,N)$ condition and, as a result, support a splitting theorem and a well-established second-order differential calculus.

In what follows, a \emph{metric measure space} (m.m.s. for short) is a triple $(Y,d,\meas)$ such that $(Y,d)$ is a proper, complete, separable, and geodesic metric space, and $\meas$ is a full-support nonnegative Radon measure on $Y$.

First, let us define $\mathrm{CD}(0,N)$ spaces, which were introduced independently by Lott and Villani in \cite{Lott-Villani_09} and Sturm in \cites{Sturm_I_06,Sturm_II_06}.

\begin{defn}[$\mathrm{CD}(0,N)$ spaces]
    Let $N\in (1,\infty)$. A m.m.s. $(Y,d,\meas)$ satisfies the \emph{$\mathrm{CD}(0,N)$ condition} if, given any pair of probability measure $\mu_0,\mu_1$ that are absolutely continuous w.r.t. $\meas$, there exists a $W_2$-geodesic $\{\mu_t\}_{0\le t\le1}$ from $\mu_0$ to $\mu_1$ such that, for every $N'\ge N$, we have the following property:
        \begin{equation*}
         \mathcal{S}_{N'}(\mu_t\mid\meas)\le t\mathcal{S}_{N'}(\mu_1\mid\meas)+(1-t)\mathcal{S}_{N'}(\mu_0\mid\meas),\quad \forall 0\le t\le 1,
    \end{equation*}
    where $\mathcal{S}_{N'}(\cdot\mid\meas)$ denotes the R\'{e}nyi entropy with parameter $N'$ associated with $\meas$.
\end{defn}

\begin{rem}
    In \cite{Lott-Villani_09}, Lott and Villani also introduced the $\mathrm{CD}(K,\infty)$ condition ($K\in\R$) for $\sigma$-finite space, while Sturm introduced the $\mathrm{CD}(K,N)$ condition ($K\in\R$, $N\in(1,\infty)$) for probability spaces. Metric measure spaces satisfying the $\mathrm{CD}(K,N)$ can be thought of as possibly singular spaces with Ricci curvature at least $K$ and dimension at most $N$ in a synthetic sense. We only introduce the simpler $\mathrm{CD}(0,N)$ condition since it will be sufficient for our purpose.
\end{rem}
Let us introduce Sobolev spaces via test plans and minimal upper gradient, following \cite{gigli_nonsmooth}. We first recall the notion of metric derivative for curves in a metric space $(Y,d)$:
\begin{defn}[Metric derivative, \cite{AmbrosioTilli_book}*{Section 4.1}]
    Given a curve $\gamma:[0,1]\to (Y,d)$ and $t\in[0,1]$, the metric derivative with respect to reference metric $d$ of $\gamma$ at $t$ is defined as
$$|\dot\gamma(t)|_d=\lim\limits_{s\to 0} \dfrac{d(\gamma(t+s),\gamma(s))}{|s|}.$$
The limit exists $\mathcal{L}^1$-a.e. on $[0,1]$.  
\end{defn}

\begin{defn}[Test plan]
    Let $e_t:C([0,1], Y)\to Y$, $e_t(\gamma)=\gamma(t)$ be the evaluation map, $t\in [0,1]$, $\bm{\pi}\in \mathcal{P}(C([0,1],Y))$ a probability measure on $C([0,1], Y)$. We say that $\bm{\pi}$ is a test plan if there exists a constant $C(\bm{\pi})>0$ such that
    \[
    (e_t)_{\sharp}\bm{\pi}\le C(\bm{\pi}) \meas, \quad \forall t\in [0,1],
    \]
    and 
    \[
    \int\int_0^1|\dot\gamma(t)|_d^2\, \di t\di\bm{\pi}(\gamma)<\infty.
    \]
    If $\gamma$ is not absolutely continuous we use the convention that $\int_0^1|\dot\gamma(t)|_d^2\, \di t=\infty$.
\end{defn}

\begin{defn}[Minimal weak upper gradient]
    Given a $\meas$-measurable function $f:Y\to \R$, a $\meas$-measurable function $G:Y\to [0,\infty]$ is called a weak upper gradient of $f$ if
    \[
    \int \left|f(\gamma(1))-f(\gamma(0))\right|\, \di\bm{\pi}(\gamma)\le \int \int_0^1G(\gamma(t))|\dot\gamma(t)|\,\di t\di\bm{\pi}(\gamma),\text{ for all test plans $\bm{\pi}$ }
    \]
    Furthermore, we say that $f$ is in the Sobolev class $S^2(Y,d,\meas)$ if it admits a weak upper gradient $G\in L^2(Y,\meas)$. We say that $G$ is a minimal weak upper gradient if $|G|\le |\tilde G|$ holds $\meas$-a.e. for any weak upper gradient $\tilde G$. The existence of a minimal weak upper gradient of $f$ is established in \cite{ambrosio_calculus_2014}*{Definition 5.11}. We denote the minimal weak upper gradient of $f$ by $|Df|_w$.
\end{defn}

\begin{defn}[Sobolev space $W^{1,2}$]
    We define the Sobolev space $W^{1,2}(Y,d,\meas)$ as the space $ L^2(Y,\meas)\cap S^2(Y,d,\meas)$ equipped with the norm 
    \[
    \|f\|_{W^{1,2}}^2\defeq \|f\|_{L^2}^2+\||Df|_w\|_{L^2}^2,
    \]
    which turn $W^{1,2}(Y,d,\meas)$ into a Banach space. To simplify notations, we write the Sobolev space as $W^{1,2}(Y)$ when the reference metric and measure are clear from the context.
\end{defn}

\begin{defn}[Locally Sobolev space $W^{1,2}_{loc}$]
   Let $\Omega\subset Y$ be an open subset. We say that $f\in L^2_{{loc}}(\Omega,\meas)$ belongs to $W^{1,2}_{loc}(\Omega,d,\meas)$ if
$\phi f \in W^{1, 2}(Y, d, \meas)$ for any $\phi \in\Lip_c(\Omega, d)$.
\end{defn}

While Ricci limit spaces satisfy the $\mathrm{CD}$ condition, it is also the case of certain non-Riemannian spaces such as non-Euclidean normed vector spaces and, more generally, certain non-Riemannian Finsler spaces \cite{Ohta_09}. To stay as close as possible to the Riemannian situation, Ambrosio--Gigli--Savar\'{e} introduced $\RCD(K,\infty)$ spaces in \cite{Ambrosio-Gigli-Savare_14} as $\mathrm{CD}(K,\infty)$ spaces whose Cheeger energy is a quadratic form (see also \cite{erbar_equivalence_2015} for the finite dimensional case).

\begin{defn}[Cheeger energy]
    The \emph{Cheeger energy} $\Ch\colon L^2(Y,\meas)\to [0,\infty]$ is defined by
    \begin{equation}\label{eq:defchee}
        \Ch(f)\defeq\begin{cases}
            \frac12\int_Y |Df|_w^2\,\di\meas &\text{if $f$ admits a weak upper gradient in $L^2(Y,\meas)$}\\
            \infty &\text{otherwise}
        \end{cases}.
    \end{equation}
\end{defn}

\begin{defn}[Infinitesimal Hilbertianity \cites{Ambrosio-Gigli-Savare_14,gigli_diff_2015} and $\RCD(0,N)$ spaces]
    A m.m.s. $(Y,d,\meas)$ is called infinitesimally Hilbertian if $\Ch$ is a quadratic form or equivalently $W^{1,2}(Y)$ is a Hilbert space. An $\RCD(0,N)$ space ($N\in(1,\infty)$) is an infinitesimally Hilbertian $\mathrm{CD}(0,N)$ space.
\end{defn}

The $\RCD(0,N)$ condition is stable under (pointed) measured Gromov--Hausdorff convergence (see \cite{GMS15}). In particular, for an open manifold $M$ with $\Ric\ge 0$, any asymptotic cone of $M$ with any limit renormalized measure is $\RCD(0,N)$. Furthermore, $\RCD(0,N)$ spaces satisfy the splitting theorem established by Gigli in \cites{Gigli13, Gigli14}.

\begin{thm}\label{thm:meas_split}\cites{Gigli13,Gigli14}
If $(Y,\dist,\meas)$ is an $\RCD(0,N)$ space (where $N\in(1,\infty)$) that contains a line, then there exists a m.m.s. $(Y',\dist',\meas')$ such that $(Y,\dist,\meas)$ is isomorphic to $(Y',\dist',\meas')\otimes (\R,d_E,\mathcal{L}^1)$, where:\\
$\bullet$ $(Y',\dist',\meas')$ is an $\RCD(0,N-1)$ space when $N\ge 2$,\\
$\bullet$ $(Y',\dist',\meas')$ is a point when $N<2$,\\
and $\R$ is equipped with Euclidean distance $d_E$ and Lebesgue measure $\mathcal{L}^1$.
\end{thm}

In addition, $\RCD(0,N)$ spaces support a second-order differential calculus. Let us recall the definition of the Laplacian and refer the reader to \cites{gigli_diff_2015,gigli_nonsmooth} for a comprehensive study.

\begin{defn}[Domain of the Laplacian $D(\Delta)$]
    Let $(Y,d,\meas)$ be an $\RCD(0,N)$ space ($N\in(1,\infty)$). A Sobolev function $\psi\in W^{1,2}(Y)$ is in the domain of the Laplacian $D(\Delta)$ if there exists $h\in L^2(Y)$ such that, for every $\phi\in \Lip_c(Y)$, we have
    $$
    \int_Y\langle \nabla \phi,\nabla \psi\rangle\di\m=-\int_Y\phi h\di\mathfrak{m},
    $$
    where $\langle \nabla \phi,\nabla \psi\rangle\defeq\frac{1}{4}(|D(\phi+\psi)|_w^2-|D(\phi-\psi)|_w^2)$. 
\end{defn}

Finally, let us define the rectifiable dimension of an $\RCD(0,N)$ spaces established in \cite{BrueSemola20}.

\begin{defn}[Rectifiable dimension]
    Let $(Y,d,\meas)$ be an $\RCD(0,N)$ space ($N\in(1,\infty)$). Given $k \in \mathbb{N}$, we define the $k$-dimensional regular set $\mathcal{R}_k(Y)$ as the set of points with the unique tangent cone $\R^k$. There exists some integer $k \le N$ such that $\meas(Y \setminus \mathcal{R}_k(Y))=0$ by Bru\'e--Semola \cite{BrueSemola20}. We call this $k$ the rectifiable dimension of $(Y,d,\meas)$.
\end{defn}

\subsection{Equivariant asymptotic cones and escape rate}

Given an $n$-manifold $M$ of $\Ric\ge 0$ and a sequence $r_i\to\infty$, after passing to a subsequence, we have pointed Gromov-Hausdorff convergence
$$(r_i^{-1}M,p)\overset{GH}\to (X,x).$$
We call $(X,x)$ an \textit{asymptotic cone} of $M$. In general, asymptotic cones of $M$ are not unique. The space $X$ naturally carries a limit renormalized measure. More precisely, for every $i\in\N$, we can define the renormalized measure $\meas_i$ on $r_i^{-1}M$ by:
\begin{equation*}
        \meas_i= \frac{\dvol}{\vol(B_{r_i}(p))}.
\end{equation*}
A \emph{limit renormalized measure} on $X$ is any measure $\meas$ on $X$ such that (passing to a subsequence if necessary) $\{(r_i^{-1}M,p,\meas_i)\}$ converges to $(X,x,\meas)$ in the (pointed) measure Gromov--Hausdorff topology. With any limit renormalized measure $\meas$, the metric measure space $(X,d,\meas)$ is an $\RCD(0,n)$ space.

Fukaya--Yamaguchi \cites{Fukaya_87,FY92} introduced the notion of equivariant Gromov--Hausdorff convergence to study the structure of fundamental groups of manifolds with uniform curvature constraints. In our context, we can use equivariant Gromov--Hausdorff convergence to study the fundamental group or covering group action on an open manifold with $\Ric\ge 0$ by blowing down the metric. Let $\widehat{M}\to M$ be a Riemannian (normal) covering map with covering group $\Gamma$, where $M$ has $\Ric\ge 0$. For any sequence $r_i\to\infty$, after passing to a subsequence, we can obtain equivariant Gromov--Hausdorff convergence
\begin{equation*}
\begin{CD}
(r_i^{-1} \widehat{M},\hat{p},\Gamma) @>GH>> (Y,y,G) \\
	@VV\pi V @VV \pi V\\
	(r_i^{-1} M,p) @>GH>> (X,x)=(Y/G,\bar{y}).
\end{CD}
\end{equation*}
We call $(Y,y,G)$ an equivariant asymptotic cone of $(\widehat{M},\Gamma)$. The limit group $G$ is a closed subgroup of the isometry group of $Y$, which is a Lie group \cite{CoNa12}. Thus $G$ itself is a Lie group. If $Y$ carries a limit renormalized measure, then $G$-action is also measure-preserving.




The notion of escape rate was first introduced by the second-named author in \cite{Pan21}; it is a quantity that measures how fast the minimal representing loops in $\pi_1(M)$ escape from bounded balls. Later, this notion was extended to any metric space with an isometric action in \cite{Pan23}*{Definition 2.17}. For covering group actions, we have:

\begin{defn}
    Let $(\widehat{M},\hat{p})\to (M,p)$ be a Riemannian (normal) covering map with an infinite covering group $\Gamma$. For each $\gamma\in \Gamma$, let $\sigma_\gamma$ be a minimal geodesic from $\hat{p}$ to $\gamma\cdot \hat{p}$. We define the escape rate of $(\widehat{M},\hat{p},\Gamma)$ by $$E(\widehat{M},\hat{p},\Gamma)=\limsup_{|\gamma|\to\infty} \dfrac{\mathrm{size}(\sigma_\gamma)}{|\gamma|},$$
   where $|\gamma|=d(\hat{p},\gamma \hat{p})=\mathrm{length}(\sigma_\gamma)$ and $\mathrm{size}(\sigma_\gamma)=\inf\{R>0 \mid \sigma_\gamma \subseteq \overline{B_R}(\Gamma\cdot \hat{p}) \}$. When choosing the minimal geodesic $\sigma_\gamma$, if there are multiple minimal geodesics from $\hat{p}$ to $\gamma \cdot \hat{p}$, we choose the one with the smallest size.
\end{defn}

Because $\mathrm{size}(\sigma_\gamma)\le |\gamma|/2$ for all $\gamma\in \Gamma$, it always holds that $E(\widehat{M},\hat{p},\Gamma)\le 1/2$.

The equivariant asymptotic geometry of $(\widehat{M},\Gamma)$ when the escape rate is non-maximal (i.e., $\not=1/2$) is well-understood thanks to the works \cites{Pan23,NPZ}. Below, we collect some of the results from \cites{Pan23,NPZ}.

\begin{prop}\cite{NPZ}*{Proposition 3.2}\label{prop:non_max_escape_rate}
Let $(M,p)$ be a complete manifold with $\Ric\ge 0$ and let $(\widehat{M},\hat{p})$ be a covering space of $(M,p)$ with covering group $\Gamma$. Suppose that $M$ is polar at infinity, then $E(\widehat{M},\hat{p},\Gamma)<1/2$.
\end{prop}

In the statement above, $M$ being polar at infinity means, for any asymptotic cone $(X,x)$ of $M$ and any point $x'\in X-\{x\}$, there is a ray from $x$ passing through $x'$.

\begin{thm}\cite{Pan23}*{Proposition C(1)}\label{thm:orbit_R}
   Let $(\widehat{M},\hat{p})\to (M,p)$ be a Riemannian covering map with covering group $\Gamma\simeq\mathbb{Z}$ and $\Ric\ge 0$. If $E(\widehat{M},\hat{p},\Gamma)\not=1/2$, then for any equivariant asymptotic cone $(Y,y,H)$ of $(\widehat{M},\Gamma)$, the orbit $Hy$ is homeomorphic to $\R$.
\end{thm}

\begin{rem}
  It follows from Theorem \ref{thm:orbit_R} that $H\simeq \R\times K$, where $K$-action fixes $y$.
\end{rem}

\section{$\RCD$ space with $S^1$-action and ray quotient}\label{sec:circle ray}
\subsection{Statements and an outline}\label{subsec:halfplane_statements}

In this entire section, we assume that $(Y,y,d,\mathfrak{m})$ satisfies the condition:
\begin{equation}\label{cond:S1_action_and_ray_quotient}
\tag{R}
\begin{gathered}
\textit{$(Y,y,d,\mathfrak{m})$ is an $\RCD(0,N)$ space with an effective isomorphic $S^1$-action such that} \\
\textit{ $(Y/S^1,\bar{y})$ is isometric to a ray $([0,\infty),0)$;} \\
\textit{we also assume that the $S^1$-action at $y$ is either free or a fixed point.}
\end{gathered}
\end{equation}

This condition enables us to approach the sharp halfplane rigidity (Theorem \ref{thm:intro_halfplane}) and the sharp ray rigidity (Theorem \ref{thm:intro_ray}) in a unified way.

\begin{lem}\label{lem:free_at_reg}
  Under the condition \eqref{cond:S1_action_and_ray_quotient}, $S^1$ acts freely on $\pi^{-1}(0,\infty)$, where $\pi\colon Y\to Y/S^1=[0,\infty)$ is the quotient map. As a consequence, $\pi^{-1}(0,\infty)$ is equivariantly homeomorphic to the open cylinder with the standard $S^1$-action $((0,\infty)\times S^1,S^1)$. Moreover:\\
  (i) if $S^1$ acts freely at $y$, then $Y$ is equivariantly homeomorphic to the half-cylinder with standard $S^1$-action $([0,\infty)\times S^1,S^1)$.\\
  (ii) if $S^1$ fixes $y$, then $Y$ is equivariantly homeomorphic to the cone with standard $S^1$-action $(C(S^1),S^1)$.
\end{lem}

\begin{proof}
     We argue by contradiction and suppose that some point $z\in Y$ with $\pi(z)>0$ has a nontrivial isotropy subgroup $H\leq S^1$. Then the $H$-action fixes every point in the orbit $S^1 z$ since $S^1$ is abelian. Let $\gamma:[0,\infty)\to Y$ be a horizontal lift of the unit speed ray in $[0,\infty)$. Then $\gamma(\pi(z))\in S^1 z$ and thus $H$ fixes $\gamma(\pi(z))$. If $h\in H$ moves a point on $\gamma$, then $\gamma$ and $h\circ \gamma$ would result in branching geodesics, which cannot occur in an $\RCD$ space \cite{Deng20}. Hence $H$ fixes every point on $\gamma$. Since $\gamma$ is any horizontal lift of the unit speed ray in $[0,\infty)$, it follows that $H$ fixes every point in $Y$; a contradiction to our assumption that the $S^1$-action is effective on $Y$. This proves that $S^1$ acts freely on $\pi^{-1}(0,\infty)$.

     Next, we consider the continuous map $$\phi\colon [0,\infty)\times S^1\to Y,\quad \phi(r,\theta)\coloneqq \theta\cdot\gamma(r).$$
     When restricted to $(0,\infty)\times S^1$, it is clear that this gives an equivariant homeomorphism between $(0,\infty)\times S^1$ and $\pi^{-1}(0,\infty)$.

     If the $S^1$-action is free at $y$, then it is also free on $\pi^{-1}(0)$. Hence the above $\phi$ becomes an equivariant homeomorphism between $[0,\infty)\times S^1$ and $Y$. If the $S^1$-action fixes $y$, then $\phi$ induces an equivariant homeomorphism between $C(S^1)$ and $Y$.
\end{proof}

\begin{notn}
   For convenience, we denote $Y_+=\pi^{-1}(0,\infty)\subseteq Y$. It is clear that $Y_+$ is an open and dense subset of $Y$. 
\end{notn}

\noindent\textbf{Goal:} The goal of this section is to show that $Y_+$, equipped with its extrinsic metric and restricted measure, is isomorphic to a warped product $\di r^2+f^2(r)\di \theta^2$ on $(0,\infty)\times S^1$ equipped with a measure $\rho(r)\di r\di v$, where $f$ and $\rho$ are a locally Sobolev function and a $\CDe(0,N)$ density on $(0,\infty)$, respectively. Thanks to the results of Mondino--Ryborz \cite{MondinoRyborz}, we will be able to translate the $\RCD(0,N)$ condition on $Y$ into distributional inequalities on $f$ and $\rho$. These distributional inequalities will be the key to obtaining the desired rigidity statements in Section \ref{sec:rigidity}.

Below, we provide a breakdown of each step and the section where it appears.

\noindent\textbf{Step 1: $Y$ has rectifiable dimension $2$ and its regular set $\mathcal{R}$ contains $Y_+$.} In Section \ref{sec:Proof_of_rectifiable_dimension_2}, we first study the equivariant tangent cone at a regular point. Under the condition \eqref{cond:S1_action_and_ray_quotient}, using ideas from the first named author \cite{Pan23}, we show that the only possibility for such a tangent cone is $\R^2$. We then show that the regular set contains $Y_+$.

\noindent\textbf{Step 2: The regular part is a $\mathcal{C}^0$-warped product.} We recover a continuous Riemannian metric on the subset $Y_+$ of $Y$.

\begin{thm}\label{thm:cont_Riem_metric}
   The metric $d$ on $Y$ comes from the metric completion of a continuous Riemannian metric 
   $$g=\di r^2 + f(r)^2 \di\theta^2$$
   on the subset $Y_+$.
\end{thm}

\noindent We prove Theorem \ref{thm:cont_Riem_metric} in two parts. First, in Section \ref{subsec:local_bilip}, we show that $d_{\lvert Y_+}$ is locally biLipschitz equivalent to the standard product distance induced by $\di r^2 +\di\theta^2$. As a result, the $\theta$-curves ($S^1$-orbits) in $Y_+$ are rectifiable. In Section \ref{subsec:cont_Riem_metric}, we obtain a (rough) Riemannian metric $ g = \di r^2 + f^2(r)\di \theta^2$, where $f(r)$ is the metric speed of the vertical curve $\theta\to (r,\theta)$. Then we further prove that $f$ is continuous on $(0,\infty)$ and $g$ indeed recovers $d_{\lvert Y_+}$.

\noindent\textbf{Step 3: From $\mathcal{C}^0$ to $W^{1,2}_{loc}$ warping function.} In Section \ref{subsec:W12_Riem_metric}, we observe that the restricted measure $\m\llcorner Y_+$ takes the form $\rho(r)\di r\di \theta$, where $\rho$ is a $\CDe(0,N)$ density on $(0,\infty)$, thanks to the results of \cite{Galaz-Garcia_18} on $\RCD$ spaces equipped with compact actions and \cite{Cavalletti-Milman_21} on one dimensional $\RCD$ spaces. Then, using the second-order calculus on $\RCD(K,N)$ spaces developed by Ambrosio, Gigli, and Savaré \cites{Ambrosio-Gigli-Savare_14,ambrosio_calculus_2014}, we show that the warping function $f$ obtained during step two satisfies $f\in W^{1,2}_{loc}(0,\infty)$.
\medskip

\noindent\textbf{Step 4: Distributional Ricci curvature bounds.} Thanks to the first three steps, $Y_+$ takes the form $((0,\infty)\times S^1, d_g, \rho(r)\di r\di \theta)$, where $g=\di r^2 + f^2(r) \di \theta^2$, and $f$ and $\rho$ are both continuous and locally Sobolev. Consequently, we are in a position to apply the tools developed by Mondino and Ryborz in \cite{MondinoRyborz}. As a result, after computing the distributional Bakry--Emery Ricci curvature, we obtain the following inequalities in the distributional sense in Section \ref{subsec:RCD_to_BE}:
\begin{equation}\label{eq:distibutional_Ric_bounds}
    -(\rho b)'\ge 0, \quad  -\rho''\rho+(\rho')^2\ge \rho^2b^2,\quad \text{ where }b\defeq (\ln(f))'.
\end{equation}
The inequalities above will be the main tool in proving the rigidity statements in Section \ref{sec:rigidity}.


\subsection{Proof of rectifiable dimension $2$}\label{sec:Proof_of_rectifiable_dimension_2}

In this subsection, we study the regular set of the space $Y$ with condition \eqref{cond:S1_action_and_ray_quotient}. We first show that $Y$ has rectifiable dimension $2$. The proof draws its idea from \cite{Pan23}*{Section 5.3}.


\begin{prop}\label{prop:rect_dim_2}
  The m.m.s.\@ $(Y,d,y,\meas)$ has rectifiable dimension $2$.
\end{prop}

\begin{proof}
   Let $k$ be the rectifiable dimension of $Y$. Let $q\in Y$ be a $k$-regular point in $Y_+$. $q$ projects to a $1$-regular point in $[0,\infty)$. For any $r_i\to 0$, we have convergence after passing to a subsequence
   \begin{equation*}
   \begin{CD}
    (r_i^{-1} Y,q,S^1) @>GH>> (\R^k,0,H) \\
	@VV\pi V @VV \pi V\\
	(r_i^{-1} [0,\infty),\pi(q)) @>GH>> (\R,0).
    \end{CD}
\end{equation*}
    Since $H$ is a closed abelian subgroup in $\mathrm{Isom}(\R^k)$ with $\R^{k}/H$ being isometric to $\R$, we see that $H\simeq \R^{k-1}$ acting by translations in $\R^k$. Together with the connectedness of the set of all equivariant tangent cones of $(Y,S^1)$ at $q$, this also implies that the equivariant tangent cone of $(Y,S^1)$ at $q$ is unique. 
    
    We need to verify $k=2$. We argue by contradiction and suppose that $k\ge 3$. We write the group elements in $S^1$ by $\theta\in [-\pi,\pi]$, with $-\pi$ and $\pi$ identified. For each $i$, let $$\theta_i=\min\{ \theta\in (0,\pi] \mid d(\theta\cdot q,q)=r_i \}.$$
    Since $r_i\to 0$, it is clear that $\theta_i\to 0$. Let $S_i\subseteq S^1$ be the symmetric subset
    $$S_i = \{ \theta\in S^1 \mid \theta\in [-\theta_i,\theta_i] \}.$$
    Now we consider the convergence
    \begin{equation}\label{eq:appx_R_orbit}
    (r_i^{-1} Y,q,\theta_i\in S_i\subseteq S^1)\overset{GH}\to (\R^k,0,g\in A\subseteq H),
    \end{equation}
    where $A\subseteq H\simeq \R^{k-1}$ is a closed symmetric subset of $H$. By construction, $A\cdot 0$ is connected with $A\cdot 0 \subseteq \overline{B_1}(0)$ and $d(g\cdot 0,0)=1$. We denote $z=g\cdot 0\in\R^k$.

    \textbf{Claim 1:} $A\cdot 0$ contains the set $\{tz\in \R^k \mid t\in[-1,1]\}$, where $tz$ is written using the linear structure of $\R^k$. Because $A$ is closed, it suffices to show that $A\cdot 0$ contains $\{ tz \mid t\in \mathbb{Q}\cap [-1,1] \}$. Let $b\in \Z_+$ and $a\in \{0,\pm 1,...,\pm (b-1)\}$. For the sequence $w_i=\frac{a}{b}v_i\in S_i$, because $bw_i=av_i$ in $S^1$, any limit $h\in H$ of $w_i$ associated with (\ref{eq:appx_R_orbit}) must satisfy $bh=ag$ in $H\simeq \R^{k-1}$. Hence $\frac{a}{b}g=h\in A$. Recall that $H$ acts on $\R^k$ by translations. Claim 1 follows.

    Thanks to Claim 1, we have $H\cdot 0\supseteq \R z$. We have supposed (by contradiction) that $H\simeq \R^{k-1}$ with $k\ge 3$. Hence $H$ contains an element $h$ such that $h\cdot 0 \not\in \R z$. Replacing $h$ by a high power $lh$ if necessary, we may assume that $d(h\cdot 0,\R z)\ge 2$. Let $u_i$ be a sequence in $S^1$ such that $u_i\overset{GH}\to h$ associated with (\ref{eq:appx_R_orbit}). It is clear that $u_i\to 0$ in $S^1$. After passing to a subsequence or replacing $h$ by $-h$ if necessary, we may assume $u_i>0$ without loss of generality.

    \textbf{Claim 2: $u_i \gg \theta_i$.} Due to $d(h\cdot 0,0)\ge 2$ and the choice of $\theta_i$, it is clear that $u_i\ge \theta_i$. Suppose that $u_i/\theta_i\to C\in [1,\infty)$ for a subsequence. We write
    $$u_i=\lfloor C \rfloor \theta_i +o_i,$$
    where $\lfloor \cdot \rfloor$ denotes the floor function and $o_i\in [0,\theta_i)\subseteq S_i$. Passing to a subsequence, we have 
    $$ \lfloor C \rfloor \theta_i \overset{GH}\to \lfloor C \rfloor g \in H,\quad o_i\overset{GH}\to \delta\in A \subseteq H$$ 
    associated with (\ref{eq:appx_R_orbit}). Hence $h=\lfloor C \rfloor g + \delta$ and
    $$d(h\cdot 0,\R z)=d((\lfloor C \rfloor g +\delta)\cdot 0,\R z)=d(\delta\cdot 0,\R z)\le d(\delta\cdot 0,0)\le 1;$$
    a contradiction to $d(h\cdot 0,\R z)\ge 2$. This proves Claim 2.

    For each $i$, we set
    $$d_i=\max\{ d(v\cdot q,q)\mid v\in [\theta_i,u_i] \}.$$
    We note that $d_i\to 0$ because $u_i\to 0$ in $S^1$.

    \textbf{Claim 3: $d_i\gg r_i$.} It is clear that $d_i\ge r_i$ because $d(v_i\cdot q,q)=r_i$. Suppose that $d_i/r_i\to C<\infty$ for a subsequence. We consider a sequence of symmetric subsets 
    $$T_i=\{ \theta \in S^1 \mid v\in[-u_i,u_i] \}$$
    and its convergence $T_i\overset{GH}\to B\subseteq H$ associated with (\ref{eq:appx_R_orbit}).
	By Claims 1 and 2, $B\cdot 0$ contains the line $\R z$. On the other hand,  $d_i/r_i\to C$ implies that $B\cdot 0$ is contained in a bounded ball $\overline{B_C}(0)$. This desired contradiction proves Claim 3.

    Next, we consider the blow-up under $d_i^{-1}\to\infty$:
    \begin{equation}\label{eq:appx_rescale}
	(d_i^{-1}Y,q,u_i\in T_i\subseteq S^1)\overset{GH}\longrightarrow (\R^k,0, h'\in B'\subseteq H'\simeq \R^{k-1}),
    \end{equation}
	where $B'$ is a closed symmetric subset of $H'$. We note that Claim 3 implies $h'y'=y'$. Because $H'\simeq \R^{k-1}$ acts freely, we conclude $h'=0$, the identity element in $H'$.

    \textbf{Claim 4.} $B'$ is a subgroup of $H'$. Let $\beta_1,\beta_2\in B'$. We choose $b_{1,i},b_{2,i}\in T_i$ such that
	$$b_{1,i}\overset{GH}\to \beta_1,\quad b_{2,i}\overset{GH}\to \beta_2$$
    associated with (\ref{eq:appx_rescale}).
	Then $\beta_1+\beta_2$ is the limit of $b_{1,i}+b_{2,i}$. If $b_{1,i}+b_{2,i}\in T_i$, then $\beta_1+\beta_2\in B'$ holds trivially. If not, we write
	$$b_{1,i}+b_{2,i}=\pm u_i + o_i,$$
	where $o_i\in [-u_i,u_i]$. Passing to a subsequence if necessary, we have
	$o_i\overset{GH}\to \delta\in B'$. Then associated with (\ref{eq:appx_rescale}), we have
	$$\beta_1+\beta_2=\lim (\pm u_i + o_i)=h'+\delta =\delta \in B'.$$
	This proves Claim 4.

    Lastly, thanks to the choice of $d_i$, $B'\cdot 0$ is contained in $\overline{B_1}(0)$ with $B'\cdot 0$ having a point on $\partial B_1(0)$. This clearly cannot happen for a subgroup $B'\subseteq H'\simeq \R^{k-1}$, which acts by translations in $\R^k$. This contradiction shows that $H\cdot 0=\mathbb{R}z$ and thus completes the proof of rectifiable dimension $2$.
\end{proof}

\begin{prop}\label{prop: reg_is_r>0}
    Let $\Reg$ be the $2$-regular set of the m.m.s.\@ $(Y,d,y,\meas)$. Then $\Reg$ contains $Y_+\coloneqq\pi^{-1}(0,\infty)$. If, in addition, $S^1$ acts freely on $\{r=0\}$, then $\Reg=Y_+$.
\end{prop}

\begin{proof}
   $Y$ has rectifiable dimension $2$ by Proposition \ref{prop:rect_dim_2}. We need to show that every point in $Y_+$ is $2$-regular. Suppose that a point $z\in Y_+$ is not $2$-regular. Because $S^1$ acts on $Y$ by isometries, any point on $S^1\cdot z$ is not $2$-regular. This curve separates $Y$ into two disconnected components, both of which have positive $\meas$-measure. Now we end in a contradiction because the $2$-regular set of $Y$ has full $\meas$-measure \cite{BrueSemola20} but should also be path connected \cite{Deng20}.

   Next, we assume that $S^1$-action is free on $\{r=0\}$ and show that every point in $\{r=0\}$ is singular. Using $S^1$-action, it suffices to show that the base point $y$ is singular. We argue by contradiction and suppose that $y$ is a regular point. We note that due to the lower semi-continuity of rectifiable dimension \cite{Kit19}, $y$ cannot be a $k$-regular point with $k\ge 3$. Blowing up at $y$ by a sequence $r_i^{-1}\to\infty$, we have
    \begin{equation}\label{eq:r=0_is_singular}
   \begin{CD}
    (r_i^{-1} {Y},y,S^1) @>\mathrm{GH}>> (\R^k,0,G) \\
	@VV\pi V @VV \pi V\\
	(r_i^{-1} [0,\infty),0) @>\mathrm{GH}>> ([0,\infty),0),
    \end{CD}
    \end{equation}
    where $k=1$ or $2$, and $\R^k/G$ is isometric to $[0,\infty)$. Given that $G$ is a non-compact subgroup of $\mathrm{Isom}(\R^k)$ by construction, the only possibility left is
    $$(\R^k,0,G)=(\R^2,0,\R\times \Z_2),$$
    where $\Z_2$ acts by a reflection with respect to the line $\R\cdot 0$. We shall rule out this case as well. We denote by $h$ a reflection element in $G$. Let $\theta_i$ be a sequence in $S^1=[-\pi,\pi]/\sim$ such that $\theta_i\overset{GH}\to h$ associated with the top row of (\ref{eq:r=0_is_singular}). We set
    $$u_i=\frac{1}{2}\theta_i\in S^1,\quad d_i=d(u_i\cdot y,y),$$
    where $d_i$ is positive because $S^1$ action is free. If the sequence $r_i^{-1}d_i$ is uniformly bounded, then $u_i\overset{GH}\to g$ for some $g\in G$ with $2g=h$; this is impossible because $h$ is a reflection. If $r_i^{-1}d_i\to \infty$ for a subsequence, we blow up at $y$ by $d_i^{-1}$ to obtain convergence
    $$(d_i^{-1} {Y},y,S^1,u_i,\theta_i)\overset{GH}\longrightarrow (\R^2,0,\R\times \Z_2,g',h')$$
    with
    $$2g'=h',\quad d(g'\cdot 0,0)=1,\quad d(h'\cdot 0,0)=0.$$
    This is also impossible in $(\R^2,0,\R\times \Z_2)$. In conclusion, any point in $\{r=0\}$ must be a singular point. This completes the proof.
\end{proof}

\subsection{Local biLipschitz estimate}\label{subsec:local_bilip}

We consider the extrinsic distance from $d$ on $Y_+$, which is still denoted by $d$ for convenience. In this subsection, we shall show that $(Y_+,d)$ is locally biLipschitz to the Euclidean distance on $(0,\infty)\times S^1$. This property plays an important role later in Section \ref{subsec:cont_Riem_metric} in recovering a continuous Riemannian metric on $Y_+$. A key input is the Abresch-Gromoll excess estimate for $\RCD$ spaces \cites{AG90,GM14}, which will be used to control the orbit locally (see Lemma \ref{lem:division_estimate}).

\begin{lem}\label{lem:uniform_regular}
   Let $K$ be a compact subset of $Y_+$. For any $z_i\in K$ and any $d_i\to 0$ such that the sequence $(d_i^{-1}Y,z_i)$ is convergent, we have
   $$(d_i^{-1}Y,z_i)\overset{GH}\longrightarrow (\R^2,0).$$
\end{lem}

\begin{proof}
  Because the rotations in $\theta$-directions $(r,\theta)\mapsto (r,\theta+u)$ are isometries for all $w\in\R$, we may assume that each $z_i$ has coordinate $(r_i,0)$ without loss of generality. After passing to a subsequence if necessary, we have $z_i\to z_0=(r_0,0)$, where $r_0>0$. Let $c:t\mapsto (t,0)$ be the horizontal ray through all $z_i$ and $z_0$. By H\"older continuity of small balls along $c|_{[r_0/2,3r_0/2]}$, there are positive constants $\alpha(N), C(N)$ such that 
  $$d_{GH}((B_s(z_i),B_s(z_0)) \le \dfrac{C(N)}{\eta} s \cdot d(z_i,z_0)^{\alpha(N)}$$
  holds for all $s$ sufficiently small and $i$ sufficiently large, where $\eta=r_0/4$  \cites{CoNa12,Deng20}. Hence for each $R>0$, we obtain estimate
  \begin{align*}
    & d_{GH}(s_i^{-1}B_{Rs_i}(z_i),B^2_R(0)) \\
    \le\, & d_{GH}(s_i^{-1}B_{Rs_i}(z_i),s_i^{-1}B_{Rs_i}(z_0)) + d_{GH}(s_i^{-1}B_{Rs_i}(z_0),B^2_R(0)) \\
    \le\, & CR\cdot d(z_i,z_0)^{\alpha(n)} + d_{GH}(s_i^{-1}B_{Rs_i}(z_0),B^2_R(0)) \to 0.
  \end{align*}
  This completes the proof.
\end{proof}

Now, we start to study the local geometry of $S^1$-orbits in $Y_+$. As usual, we write elements in $S^1$ by $\theta\in [-\pi,\pi]$.

\begin{lem}\label{lem:deviation_estimate}
   Let $\epsilon>0$ and $y_0=(r_0,0)\in Y_+$, where $r_0>0$. There is a small $\delta>0$ such that for all $(r,\theta)\in B_\delta(y_0)$, it holds that
   $$d((r,\theta/2),\sigma)\le \epsilon\cdot d((r,0),(r,\theta)),$$
   where $\sigma$ is a minimal geodesic from $(r,0)$ to $(r,\theta)$. 
\end{lem}

\begin{proof} The proof is by a standard contradiction argument. Suppose that there are $\epsilon_0>0$, $\delta_i\to 0$, and a sequence $z_i=(r_i,\theta_i)\in B_{\delta_i}(y_0)$, where $\theta_i\not=0$, such that
\begin{equation}\label{eq:deviation_estimate}
\dfrac{d((r_i,\theta_i/2),\sigma_i)}{d((r_i,0),(r_i,\theta_i))}\ge \epsilon_0
\end{equation}
where $\sigma_i$ is a minimal geodesic from $(r_i,0)$ to $(r_i,\theta_i)$. Let $d_i=d((r_i,0),(r_i,\theta_i))\to 0$. After passing to a subsequence if necessary, by Lemma \ref{lem:uniform_regular}, we consider
$$(d_i^{-1} Y,(r_i,0),(r_i,\theta_i),S^1)\overset{GH}\longrightarrow (\R^2,(0,0),(0,1),\R).$$
Since $(r_i,0)$ and $(r_i,\theta_i)$ belong to the same orbit, this property is passed to the limit. The limit $\R$-action is by translation in the second coordinate. Associated with this convergence, we have $(r_i,\theta_i/2)$ converging to $(0,1/2)\in \R^2$ and $\sigma_i$ converging to the unique segment from $(0,0)$ to $(0,1)$. In particular, the limit geodesic goes through $(0,1/2)$. This yields a contradiction to 
(\ref{eq:deviation_estimate}).
\end{proof}

\begin{lem}\label{lem:division_estimate}
  Let $\epsilon>0$ and $y_0=(r_0,0)\in Y_+$, where $r_0>0$. There is $\delta>0$ such that for all $(r,\theta)\in B_\delta(y_0)$ and all $k\in \mathbb{N}$, it holds that
  $$2^k \cdot d((r,0),(r,\theta/2^k))\le (1+\epsilon) \cdot d((r,0),(r,\theta)).$$
\end{lem}

\begin{proof}
   Let $\delta\in (0,1/10)$ such that $\overline{B_{2\delta}}(y_0)$ is contained in the regular set. We may further shrink $\delta$ along the proof.

   Let $(r,\theta)\in B_\delta(y_0)$. For each $k$, we denote
   $$\delta_k= d((r,0),(r,\theta/2^k))$$
   and $l_k$ the distance from the point $(r,\theta/2^{k+1})$ to $\sigma_k$, a minimal geodesic from $(r,0)$ to $(r,\theta/2^k)$. From Lemma \ref{lem:deviation_estimate}, we have seen that $l_k \ll \delta_k$ as $k\to\infty$. By Abresch--Gromoll excess estimate \cites{AG90,GM14}, we obtain
   $$d((r,\theta/2^{k+1}),(r,\theta/2^{k}))+d((r,\theta/2^{k+1}),(r,0))-d((r,\theta/2^{k}),(r,0))\le C(N)\cdot l_k^{\frac{N}{N-1}}.$$
   That is,
   \begin{equation}\label{eq:half_estimate_by_AG}
   2\delta_{k+1}\le \delta_k + C(N) l_k^{\frac{N}{N-1}}= \delta_k \left( 1+ C(N) \dfrac{l_k}{\delta_k}\cdot l_k^{\frac{1}{N-1}} \right).
   \end{equation}
   By Lemma \ref{lem:deviation_estimate}, we can shrink $\delta>0$ such that
   $$\dfrac{d((r,\theta/2),\sigma)}{d((r,0),(r,\theta))}\le \min\{C(N)^{-1},0.1\},$$
   for all $(r,\theta) \in B_\delta(y_0)$, where $\sigma$ denotes a minimal geodesic from $(r,0)$ to $(r,\theta)$.
   Then  
   $$ \dfrac{l_k}{\delta_k}\le C(N)^{-1},\quad l_k^{\frac{1}{N-1}}\le 0.1$$
   hold for all $k\in \mathbb{N}$. This gives a rough estimate 
   $$\delta_{k+1} \le  \delta_k \cdot \dfrac{1.1}{2}.$$
   Hence
   $$\delta_k \le \left( \dfrac{1.1}{2} \right)^k \delta_0$$
   for all $k\in \mathbb{N}$.
   
   To improve this estimate, we iterate the inequality (\ref{eq:half_estimate_by_AG}) and derive
   $$2^k \delta_{k}\le \delta_0 \cdot \prod_{i=0}^{k-1} \left( 1+ C(N) \dfrac{l_i}{\delta_i}\cdot l_i^{\frac{1}{N-1}} \right).$$
   After applying $\ln$ function to the product on the right side, we estimate 
   $$\sum_{i=0}^{k-1} \ln\left( 1+C(N) \dfrac{l_i}{\delta_i}\cdot l_i^{\frac{1}{N-1}}\right)\le \sum_{i=0}^{k-1} C(N) \dfrac{l_i}{\delta_i}\cdot l_i^{\frac{1}{N-1}} \le \sum_{i=0}^{k-1} \delta_i^{\frac{1}{N-1}}=\delta_0\sum_{i=0}^{k-1} (0.55^\frac{1}{N-1})^i.$$
   Setting 
   $$c=\sum_{i=0}^\infty (0.55^\frac{1}{N-1})^i<\infty,$$
   we have 
   $$2^k \delta_k \le \delta_0 \cdot e^{\delta_0 c}.$$
   If we choose $\delta>0$ sufficiently small in the beginning such that
   $$e^{\delta c} \le 1+ \epsilon,$$
   then the desired estimate follows.
\end{proof}

 Lemma \ref{lem:division_estimate} directly implies that $\theta$-curves (the $S^1$-orbits) in $Y_+$ are rectifiable. Indeed, we obtain a control on their lengths as shown by the corollary below.

\begin{cor}\label{cor:v_curve_length}
   Let $\epsilon>0$ and let $y_0=(r_0,0)\in Y_+$, where $r_0>0$. There is $\delta>0$ such that
   $$\mathrm{length}_d(\sigma_r)\le (1+\epsilon)\cdot d((r,0),(r,\theta))$$
   for all $(r,\theta)\in B_\delta(y_0)$, where $\sigma_r$ is the vertical curve from $(r,0)$ to $(r,\theta)$, that is,
   $\sigma_r(t)=(r,t)$.
\end{cor}

\begin{proof}
   By Lemma \ref{lem:division_estimate}, we have
   $$2^k \cdot d((r,0),(r,\theta/2^k))\le (1+\epsilon) \cdot d((r,0),(r,\theta)).$$
   for all $k\in \mathbb{N}$. Using isometric $\R$-action, we interpret the left-hand side as 
   $$\sum_{j=0}^{2^k-1} d((r,j\theta/2^k),(r,(j+1)\theta/2^k)).$$
   Letting $k\to\infty$, we obtain the desired estimate.
\end{proof}

Next, we prove a local Pythagorean estimate around any regular point.

\begin{lem}\label{lem:pyth_estimate}
   Let $\epsilon>0$ and $y_0=(r_0,0)\in Y_+$, where $r_0>0$. There is $\delta>0$ such that for all $(r,\theta)\not=(r',\theta')\in B_\delta(y_0)$, it holds that
  $$1-\epsilon \le \dfrac{\left(d((r,\theta),(r,\theta'))^2+ d((r,\theta'),(r',\theta'))^2 \right)^{1/2}}{d((r,\theta),(r',\theta'))}\le 1+\epsilon.$$
\end{lem}

\begin{proof}
We argue by contradiction and suppose that there are $\epsilon_0>0$, contradicting sequences $\delta_i\to 0$ and 
$(r_i,\theta_i)\not=(r'_i,\theta'_i)$ in $B_{\delta_i}(y_0)$ but
$$\left| \dfrac{\left(d((r_i,\theta_i),(r_i,\theta'_i))^2+ d((r_i,\theta'_i),(r'_i,\theta'_i))^2 \right)^{1/2}}{d((r_i,\theta_i),(r'_i,\theta'_i))} -1 \right|\ge \epsilon_0.$$
Let $d_i=d((r_i,\theta_i),(r'_i,\theta'_i))\to 0$. We denote $\gamma_i$ the horizontal ray through $(r'_i,\theta'_i)$. After passing to a subsequence if necessary, we apply Lemma \ref{lem:uniform_regular} and consider the convergence
$$(d_i^{-1}Y,(r_i,\theta_i),(r'_i,\theta'_i),\R,\gamma_i)\overset{GH}\longrightarrow (\R^2,(0,0),z=(a,b),\R,\gamma_\infty)$$
with $d_{\R^2}((0,0),(a,b))=1$. We understand the limit $\R$-action as translations in the second coordinate of $\R^2$, and $\gamma_\infty$ as the horizontal line through $(a,b)$. Since
$(r_i,\theta'_i)$ and $(r'_i,\theta'_i)$ belong to the same horizontal ray $\gamma_i$, it holds that
$$d((r_i,\theta_i),(r_i,\theta'_i))\le d((r_i,\theta_i),(r'_i,\theta'_i))+d((r'_i,\theta'_i),(r_i,\theta'_i))\le 2d_i.$$
Therefore, $(r_i,\theta'_i)$ subconverges to a limit point $w$ in $\R^2$ associated to the above blow-up sequence. By construction, $w$ satisfies\\
(i) $w$ and $(0,0)$ belong to the same $\R$-orbit,\\
(ii) $w$ and $z=(a,b)$ belong to the same horizontal line $\gamma_\infty$.\\
These properties force $w=(0,b)$. Hence by Pythagorean theorem in $\R^2$, we have
$$d_{\R^2}((0,0),w)^2 + d_{\R^2}(w,z)^2 = d_{\R^2}((0,0),z)^2,$$
which leads to the desired contradiction. This completes the proof.
\end{proof}

We are ready to prove the local biLipschitz estimate around any point in $Y_+$. Because the statement is local, we will identify a neighborhood of $(r_0,0)$ in $Y_+=(0,\infty)\times S^1$, expressed in $(r,\theta)$-coordinate, with a neighborhood in $\R^2$.

\begin{prop}\label{thm:local_bilip}
   Given $y_0=(r_0,0)\in Y_+$, where $r_0>0$, there is a neighborhood $U$ of $y_0$ and positive constants $C_1$, $C_2$ such that
   $$ C_1\cdot d(z,w) \le \| z-w \|\le C_2\cdot d(z,w),$$
   for all $z,w\in U$, where $\| \cdot \|$ denotes the Euclidean norm on $\R^2$ with $(r,\theta)$-coordinate.
\end{prop}

\begin{proof}
   Let $\epsilon=0.1$. We choose $\eta>0$ such that
   $$U\coloneqq(r_0-\eta,r_0+\eta)\times (-\eta,\eta)\subset B_{\delta/2}(y_0)\subseteq  Y_+,$$
   where $\delta>0$ is chosen so that the conclusions of Lemmas \ref{lem:division_estimate} and \ref{lem:pyth_estimate} are satisfied.
   We select $C_1\in (0,1/2)$ and $C_2\in (2,\infty)$ such that
   $$C_1\le  \dfrac{| \theta |}{d((r,0),(r,\theta))}\le C_2$$
   for all $(r,\theta)\in U$ with $r\in [r_0-\eta,r_0+\eta]$ and $\theta\in [\eta/2,\eta]\cup [-\eta,-\eta/2]$.

   We first consider the special case when $z$ and $w$ belong to the same $S^1$-orbit; in other words, they share the same $r$-coordinate. Applying an isometry from $\R$-action, we can assume
   $$z=(r,0),\quad w=(r,\theta)$$
   without loss of generality.
   If $v \in [\eta/2,\eta]\cup [-\eta,-\eta/2]$, then clearly we have 
   $$\|z-w\|=|\theta| \le C_2\cdot d((r,0),(r,\theta))=C_2\cdot d(z,w);$$
   $$\|z-w\|=|\theta| \ge C_1\cdot d((r,0),(r,\theta))=C_1\cdot d(z,w).$$
   If $\theta\in (-\eta/2,\eta/2)$, we choose a positive integer $k$ such that $2^{k} \theta \in [\eta/2,\eta]\cup [-\eta,-\eta/2]$, then by triangle inequality
   $$\dfrac{\|z-w\|}{d(z,w)}=\dfrac{\|2^k \theta\|}{2^k d((r,0),(r,\theta))}\le \dfrac{\|2^k \theta\|}{d((r,0),(r,2^k \theta))}\le C_2.$$
   For the other direction, we apply Lemma \ref{lem:deviation_estimate} and estimate
   \begin{align*}
      \dfrac{\|z-w\|}{d(z,w)}=\dfrac{|2^k \theta|}{2^k d((r,0),(r,\theta))}\ge \dfrac{|2^k \theta|}{(1+\epsilon) d((r,0),(r,2^k \theta)) } \ge \dfrac{C_1}{1+\epsilon}.
   \end{align*}

   For a general pair of distinct points 
   $$z=(r,\theta),\quad w=(r',\theta'),$$
   we can apply Lemma \ref{lem:pyth_estimate}. Indeed,
   \begin{align*}
       \|z-w\|&= \left( |r-r'|^2 + |\theta-\theta'|^2 \right)^{1/2} \\
       &\le \left( d((r',\theta'),(r,\theta'))^2 + C_2^2 d((r,\theta'),(r,\theta))^2 \right)^{1/2} \\
       &\le C_2 \left( d((r',\theta'),(r,\theta')) + d((r,\theta'),(r,\theta))^2 \right)^{1/2} \\
       &\le C_2(1+\epsilon)\cdot d(z,w).
   \end{align*}
   The other direction can be proved similarly by using Lemma \ref{lem:pyth_estimate} to obtain
   $$\|z-w\| \ge \dfrac{1-\epsilon}{1+\epsilon}\cdot C_1 \cdot d(z,w).$$
   We complete the proof.
\end{proof}

\subsection{Recovering a \texorpdfstring{$C^0$}{C0} Riemannian metric on the $Y_+$}\label{subsec:cont_Riem_metric}

We prove Theorem \ref{thm:cont_Riem_metric} in this subsection.

For each $r>0$, we define an inner product $g$ at $(r,0)$ (and thus also at $(r,\theta)$ for all $\theta\in S^1$ due to $S^1$-isometries) by declaring an orthogonal basis $\{\partial_r,\partial_\theta\}$ with norm
$$\| \partial_\theta \|_g = \dfrac{\mathrm{length}(\sigma_r|_{[0,t]})}{t},\quad \|\partial_r\|_g=1,$$
where $\sigma_r$ is the vertical curve $\sigma_r(t)=(r,t)$. We note that
$\mathrm{length}(\sigma_r|_{[0,t]})$ is a linear function for $t\ge 0$, thus $\|\partial_\theta\|_g$ above is well-defined. This provides a positive function 
$$f(r)\coloneqq\| \partial_\theta \|_g \text{ at } (r,0).$$
Thanks to Lemma \ref{lem:cont_warp} below, we obtain a continuous Riemannian metric 
$$g=\di r^2 + f(r)^2 \di \theta^2 \quad \text{on } Y_+=\{r>0\}.$$

\begin{lem}\label{lem:cont_warp}
   The warping function $f$ is continuous on $(0,\infty)$.
\end{lem}

\begin{proof}
   Let $\epsilon>0$ and $r_0>0$. We choose $\delta>0$ such that Corollary \ref{cor:v_curve_length} holds. For each $r\in(r_0-\delta/2,r_0+\delta/2)$, we denote $\sigma_r$ the vertical curve at $(r,0)$, that is,
   $$\sigma_r(t)=(r,t), \quad t\in \R.$$
   We fix a small $t_0>0$ such that the image of $\sigma_r|_{[0,t_0]}$ is contained in $B_\delta(y_0)$ for all $r\in(r_0-\delta/2,r_0+\delta/2)$. Next, by the continuity of the positive function $r\mapsto d((r,0),(r,t_0))$, we choose a small $\delta'\in (0,\delta/2)$ such that
   $$1-\epsilon \le \dfrac{d((r,0),(r,t_0))}{d((r_0,0),(r_0,t_0))}\le 1+\epsilon$$
   for all $r\in(r_0-\delta',r_0+\delta')$.
   Together with the estimate
   $$d((r,0),(r,t_0))\le \mathrm{length}(\sigma_r|_{[0,t_0]})\le (1+\epsilon)\cdot d((r,0),(r,t_0))$$
   from Corollary \ref{cor:v_curve_length}, we obtain
   $$\dfrac{1-\epsilon}{1+\epsilon}\le \dfrac{f(r)}{f(r_0)}\le (1+\epsilon)^2$$
   for all $r\in(r_0-\delta',r_0+\delta')$. This shows the continuity of $f$ at $r_0.$
\end{proof}

\begin{lem}\label{lem:norm=limit}
   Let $y_0=(r_0,0)\in Y_+$, where $r_0>0$, and $W\in T_{y_0}\R^2$, then
   $$\|W\|_g = \lim\limits_{t\to 0}\dfrac{d(y_0,y_0+tW)}{|t|}$$
\end{lem}

\begin{proof}
   Let $\epsilon>0$. We write $W=(a,b)\in T_{y_0}\R^2-\{(0,0)\}$. Then
   $$\| W \|^2_g = a^2 + b^2 f(r_0)^2.$$
   Thanks to Lemma \ref{lem:pyth_estimate}, when $|t|$ is sufficiently small, we have
   $$\left|\dfrac{t^2 a^2 + d(y_0,(r_0,tb))^2}{d(y_0,y_0+tW)^2}-1\right| \le \epsilon.$$
   We denote $\sigma$ the vertical curve $\sigma(t)=(r_0,bt)$, where $t\in [0,1]$, and apply Lemma \ref{cor:v_curve_length}, then 
   $$ (1-\epsilon)\cdot \mathrm{length}(\sigma|_{[0,t]}) \le d(y_0,(r_0,tb)) \le \mathrm{length}(\sigma|_{[0,t]})$$
   when $|t|$ is sufficiently small. Recall that by construction of the Riemannian metric $g$, we have
   $$\mathrm{length}(\sigma|_{[0,t]})=tb\|\partial_v\|_g=tb\cdot f(r_0).$$
   Together, they yield
   $$\left|\dfrac{t^2 (a^2 + b^2f(r_0)^2)}{d(y_0,y_0+tW)^2}-1\right| \le \epsilon.$$
   Letting $t\to 0$, we prove the desired equality.
\end{proof}

To prove that the defined Riemannian metric $g$ coincides with the (extrinsic) distance $d$ on $Y_+$, 
we will use the following fact about metric derivatives later in our proof: for any Lipschitz curve $\gamma:[0,1] \to (Y,d)$, the metric derivative $|\dot\gamma(t)|_d$ exists almost everywhere and
$$\mathrm{length}_d(\gamma)=\int_0^1 |\dot\gamma(t)|_d dt.$$
See \cite{AmbrosioTilli_book}*{Theorem 4.1.6}. We also denote $\gamma'(t)$ the standard derivative of $\gamma(t)$, since we can view $\gamma$ as a curve in $\R^2$.

\begin{lem}\label{lem:metric_deriv=usual_deriv}
  Let $\gamma:[0,1]\to (Y_+,d)$ be a Lipschitz curve. Let $t\in [0,1]$ such that both $\gamma'(t)$ and $|\dot\gamma(t)|_d$ exist. Then
  $$|\dot\gamma(t)|_d=\| \gamma'(t) \|_g.$$
\end{lem}

\begin{proof}
   Let $\epsilon>0$. We fix a $t\in[0,1]$ such that both $\gamma'(t)$ and $|\dot\gamma(t)|_d$ exist. We write $y_0=\gamma(t)$ and $W=\gamma'(t)\in T_{y_0}\R^2$.
   
   We first prove $|\dot\gamma(t)|_d\le \|\gamma'(t)\|_g$. We pick $\delta>0$ such that
   $$\| \gamma(t+s)-(\gamma(t)+sW) \|_{\R^2}\le \epsilon |s|$$
   for all $s\in (-\delta,\delta)$. By Proposition \ref{thm:local_bilip}, after further shrinking $\delta$ if necessary, we have
   $$d(\gamma(t+s),y_0+sW)\le C_1^{-1}\epsilon |s|.$$
   Then
   \begin{align*}
   \dfrac{d(\gamma(t+s),y_0)}{|s|} &\le \dfrac{d(\gamma(t+s),y_0+sW)}{|s|}+\dfrac{d(y_0+sW,y_0)}{|s|}\\
   &\le C_1^{-1}\epsilon + \dfrac{d(y_0+sW,y_0)}{|s|}.
   \end{align*}
   Let $s\to 0$ and apply Lemma \ref{lem:norm=limit}, then we obtain
   $$|\dot\gamma(t)|_d\le C_1^{-1}\epsilon + \|W\|_g.$$
   This proves one direction $|\gamma'|_d(t)\le \|\gamma'(t)\|_g$.

   For the other direction, we estimate
   $$\dfrac{d(y_0,y_0+sW)}{|s|}\le \dfrac{d(y_0,\gamma(t+s))}{|s|}+\dfrac{d(\gamma(t+s),y_0+sW)}{|s|}\le \dfrac{d(y_0,\gamma(t+s))}{|s|} + C_1^{-1}\epsilon.$$
   Letting $s\to 0$, we have
   $$\|W\|_g\le |\dot\gamma(t)|_d+C_1^{-1}\epsilon.$$
   This completes the proof.
\end{proof}

According to \cite{Burtscher15}*{Theorem 4.1}, the continuous Riemannian metric $g$ on $Y_+$ induces a distance function $d_g$ on $Y_+$ that coincides with the manifold topology, where $d_g$ is defined by 
$$d_g(z,w)=\inf\{ \mathrm{length}_g(\gamma)\mid \gamma \text{ is a piecewise smooth curve from $z$ to $w$}\}.$$ 
With Lemma \ref{lem:metric_deriv=usual_deriv}, we easily conclude: 


\begin{prop}\label{prop:d=d_g_on_reg}
    $d=d_g$ on $Y_+$.
\end{prop}

\begin{proof}
   We will handle the two cases in the condition \eqref{cond:S1_action_and_ray_quotient} separately.

   We first consider the case that $S^1$-action is free on $\{r=0\}$. In this case, $\{r=0\}$ coincides with $\mathcal{R}$, thus is strongly convex by the H\"older continuity of tangent cones in the interior of $\sigma$ \cite{Deng20}. Let $z,w\in Y_+$ and let $\gamma:[0,1]\to Y_+$ be a $d$-geodesic from $z$ to $w$ of arc-length parametrization. Then by \cite{AmbrosioTilli_book}*{Theorem 4.1.6} and Lemma \ref{lem:metric_deriv=usual_deriv}, 
   $$d(z,w)=\mathrm{length}_d(\gamma)=\int_0^1 |\dot\gamma(t)|_d \di t=\int_0^1 \|\gamma'(t)\|_g\, \di t \ge d_g(z,w).$$
   For the other direction, let $\epsilon>0$ and $\gamma:[0,1]\to Y_+$ be a piecewise smooth curve of constant speed such that 
   $$\mathrm{length}_g(\gamma)-\epsilon \le d_g(z,w)\le \mathrm{length}_g(\gamma).$$
   Thanks to Theorem \ref{thm:local_bilip}, $\gamma$ is locally Lipschitz with respect to $d$ and thus its metric derivative exists almost everywhere. Then again by \cite{AmbrosioTilli_book}*{Theorem 4.1.6} and Lemma \ref{lem:metric_deriv=usual_deriv}, 
   $$d_g(z,w) \ge \mathrm{length}_g(\gamma)-\epsilon = \int_0^1 \|\gamma'(t)\|_g \di t-\epsilon = \int_0^1 |\dot\gamma(t)|_d \di t-\epsilon\ge d(z,w)-\epsilon.$$

   Next, we deal with the second case, where $S^1$-action fixes $\{r=0\}=\{y\}$.  Let $z,w\in Y_+$ and let $\gamma:[0,1]\to Y$ be a $d$-geodesic from $z$ to $w$ of arc-length parametrization. If $\gamma$ is contained in $Y_+$, then by the same proof in the first case, we obtain $d(z,w)\ge d_g(z,w)$.
   If $\gamma$ passes through $\{r=0\}$ at $\gamma(t_0)=y$, then both pieces $\gamma|_{[0,t_0)}$ and $\gamma|_{(t_0,1]}$ are contained in $Y_+$. Hence the same estimate implies $d(z,w)\ge d_g(z,w)$. The other direction of the inequality follows from the same proof in the first case.
\end{proof}

Now Theorem \ref{thm:cont_Riem_metric} follows directly from Proposition \ref{prop:d=d_g_on_reg}.

\begin{proof}[Proof of Theorem \ref{thm:cont_Riem_metric}]
    Because the set $Y_+$ is an open and dense subset in $(Y,d)$, the metric completion of $(Y_+,d=d_g)$ is isometric to $(Y,d)$.
\end{proof}

\subsection{Improving to a \texorpdfstring{$C^0 \cap W^{1,2}_{loc}$}{continuous and Sobolev} Riemannian metric}\label{subsec:W12_Riem_metric}

As mentioned earlier at the end of Section \ref{subsec:halfplane_statements}, our proof of the rigidity statements in Section \ref{sec:rigidity} will rely on the equivalence between the $\RCD(0,N)$ condition and the non-negativity of the Bakry-Emery Ricci curvature in a distributional sense, established by Mondino and Ryborz in \cite{MondinoRyborz}. In order to apply their results, we need a Riemannian metric and measure of sufficient regularity.

First of all, we show that the restricted measure $\meas\llcorner Y_+$ satisfies the following regularity and invariance property.

\begin{lem}\label{lem:measure density}
    The restriction of the measure $\meas$ to $Y_+$ satisfies $\m{\llcorner Y_+} = \rho(r)\di r\di \theta$, where $\rho\colon(0,\infty)\to(0,\infty)$ is a $\mathrm{CD}(0,N)$ density, i.e. $\rho^{\frac{1}{N-1}}$ is concave.
\end{lem}

\begin{proof}
     Thanks to \cite{Galaz-Garcia_18}*{Theorem 1.1}, the quotient m.m.s.\@ $([0,\infty),d_E,\m^*)$ satisfies the $\RCD(0,N)$ condition, where $\m^*\coloneqq\pi_\#\m$ is the quotient measure. Thus, as a result of \cite{Cavalletti-Milman_21}*{Theorem A.2}, $\m^*=\rho(r)\di r$, where $\rho\colon[0,\infty)\to[0,\infty)$ is a $\CDe(0,N)$ density. Now, we fix $R>0$ and consider the restricted measures
     \[
    \m_R\coloneqq\m{\llcorner\pi^{-1}[0,R]},\qquad \m^*_R\coloneqq\m^*{\llcorner[0,R]}.
     \]
     It is clear that both measures are finite with equal total mass, that $\m_R$ is invariant under the $S^1$ action, and that $\pi_{\#}\m_R=\m^*_R$. Therefore, thanks to \cite{Galaz-Garcia_18}*{Theorem 3.2}, $\m_R$ is the lift of $\m^*_R$. More precisely, given any Borel subset $B\subset Y$, we have the following property
     \[
     \meas_R(B)=\int_{[0,\infty)}\nu_r(B)\di\m^*_R(r),\text{ where }\nu_r(B)=(2\pi)^{-1}\int_{S^1}\delta_{\theta\cdot\gamma(r)}(B)\di\theta,
     \]
     and where $\gamma\colon[0,\infty)\to Y$ is some fixed horizontal lift of $[0,\infty)$. Thanks to \ref{lem:free_at_reg}, we identified $Y_+$ with $(0,\infty)\times S^1$ equipped with the standard $S^1$ action. Under this identification, for every $r>0$, we have 
     \[
     \nu_r(B) = (2\pi)^{-1}\vol_{S^1}\left(\left\{\theta\in S^1, (r,\theta)\in B\right\}\right).
     \]
     As a result, $\m$ and $(2\pi)^{-1}\rho(r)\di r\di \theta$ coincide on every Borel subset $B\subset Y_+\cap\pi^{-1}[0,R]$. Replacing $\rho$ by $(2\pi)^{-1}\rho$ and letting $R$ go to $\infty$ concludes the proof.
\end{proof}
\begin{rem}
    Since the concavity of $\rho^{\frac{1}{N-1}}$ implies the concavity of $\rho^{\frac{1}{N}}$ and the latter is sufficient for our purposes. We will only use the concavity of $\rho^{\frac{1}{N}}$ in the sequel.
\end{rem}

\begin{rem}\label{rem:orbits_have_measure_zero}
    Another immediate consequence of \cite{Galaz-Garcia_18}*{Theorem 3.2} is that orbits and horizontal rays have measure zero.
\end{rem}

Next, our goal is to prove that the continuous warping function obtained in Theorem \ref{thm:cont_Riem_metric} is actually locally Sobolev on $Y_+$.

\begin{prop}\label{prop: warping function is sobolev}
    There exists a positive $W^{1,2}_{loc}$ function $f\colon(0,\infty)\to(0,\infty)$ such that $(Y_+,d_{\lvert Y_+})$ is isometric to $((0,\infty)\times S^1,\di r^2+f^2(r)\di\theta^2)$.
\end{prop}

\begin{rem}
    As an $\RCD(0,N)$ space, $(Y,y,d,\mathfrak{m})$ is doubling and supports a local Poincaré inequality. Therefore, thanks to \cite{cheeger_differentiability_1999}*{Theorem 6.1}, the minimal weak upper gradient of a Lipschitz function $\phi\in\Lip(Y,d)$ satisfies $|D\phi|_w = \lip(\phi)$ $\meas-$a.e., where:
    $$
    \lip(\phi)\colon x\in Y\mapsto \limsup_{y\to x}\frac{|\phi(y)-\phi(x)|}{d(x,y)}\in[0,\infty),
    $$
    denotes the local Lipschitz constant of $\phi$.
\end{rem}

\begin{rem}
    For $\RCD(K,N)$ spaces ($K\in\R$), the minimal weak upper gradient construction developed by Ambrosio, Gigli, and Savaré coincides with Cheeger's construction in \cite{cheeger_differentiability_1999} (see \cite{ambrosio_calculus_2014}*{Remark 4.7}).
\end{rem}

We recall that, thanks to Lemma \ref{lem:cont_warp} and Proposition \ref{prop:d=d_g_on_reg}, we identified $(Y_+,d_{\lvert Y_+})$ with $((0,\infty)\times S^1, d_g)$, where
\[
g\coloneqq\di r^2+f^2(r)\di\theta^2,
\]
and $f\colon(0,\infty)\to(0,\infty)$ is a continuous function. In order to prove Proposition \ref{prop: warping function is sobolev}, we will need a few preliminary lemmas. Our first lemma is a classical result relating the local Lipschitz constant to the gradient norm with respect to our continuous Riemannian metric $g$; we provide a proof for completeness.

\begin{lem}\label{lem: weak gradient vs lip}
    If $\phi\in\Lip(Y)$, then $\phi$ is differentiable $\meas$-a.e.\ on $Y_+$. Moreover, if $\phi$ is differentiable at $x\in Y_+$, then $\lip(\phi)(x) = |\nabla_g\phi|(x)$.
\end{lem}

\begin{proof}
    Let $\phi\in\Lip(Y)$ and observe that $\phi$ is $d_g$-Lipschitz on $Y_+=(0,\infty)\times S^1$. By \cite{burtscher_length_2015}*{Theorem 4.5}, $\phi$ is locally Lipschitz w.r.t.\ the product metric $\di r^2+\di\theta^2$. Moreover, thanks to Lemma \ref{lem:measure density}, we know that $\meas\llcorner Y_+\sim\di r\di\theta$. Therefore, $\phi$ is differentiable $\meas$-a.e.\ on $Y_+$ by Rademacher's theorem.
    
    Now, assume that $\phi$ is differentiable at $x\in Y_+$ and let us show that $\lip(\phi)(x) = |\nabla_g\phi|(x)$. We are going to compare $g$ with $g_x=\di r^2+f^2(\pi(x))\di\theta^2$. Given $\epsilon>0$, we consider a compact neighborhood $K\subset Y_+$ of $x$ such that
    \begin{equation}\label{eq:g_y vs g_x}
            \forall y\in K,\ (1-\epsilon)^2g_x\le g_y\le (1+\epsilon)^2g_x.
    \end{equation}
    Without loss of generality, we may also assume that $K$ is strongly convex w.r.t.\@ $g_x$. In particular, \eqref{eq:g_y vs g_x} implies that, for all $y\in K$, we have $d_g(x,y)\le(1+\epsilon)d_{g_x}(x,y)$. For the other inequality, let $\delta>0$ such that $B_d(x,\delta)\subset K$ and assume that $y\in B_d(x,\delta)$.  Since $Y$ is a geodesic space, there exists a geodesic $\gamma$ from $x$ to $y$. In particular, since $d(x,\gamma(t))=td(x,y)$, $\gamma$ stays in $B_d(x,\delta)$ which is a subset of $Y_+$. Thanks to Proposition \ref{prop:d=d_g_on_reg}, $\gamma$ is a $g$-geodesic and, as a result, we have $(1-\epsilon)d_{g_x}(x,y)\le \mathcal{L}_g(\gamma)=d_g(x,y)$. Summarizing, we have
    \[
    \forall\, y\in B_d(x,\delta), \qquad (1-\epsilon)d_{g_x}(x,y)\le d_g(x,y)=d(x,y)\le (1+\epsilon)d_{g_x}(x,y).
    \]
    Therefore, $(1-\epsilon)\lip_{{g_x}}(\phi)(x)\le\lip(\phi)(x)\le(1+\epsilon)\lip_{{g_x}}(\phi)(x)$. However, since $g_x$ is a smooth Riemannian metric, we have $\lip_{{g_x}}(\phi)(x)=|\nabla_{g_x}\phi|(x)=|\nabla_g\phi|(x)$. Thus, letting $\epsilon$ go to $0$ concludes the proof.
\end{proof}

\begin{rem}\label{rem: v lip}
    If $\psi\in\mathcal{C}^{\infty}( Y_+)$, then $\psi$ is locally Lipschitz with respect to the product metric $\di r^2+\di\theta^2$. In particular, $\psi\in\Lip_{loc}( Y_+)$ thanks to \cite{burtscher_length_2015}*{Theorem 4.5}. Furthermore, we may compute the gradient of $\psi$ as follows
    \begin{equation}\label{eq:gradient of psi}
        \nabla_g\psi=\partial_r\psi\cdot \partial_r+f^{-2}(r)\partial_{\theta}\psi\cdot \partial_{\theta}.
    \end{equation}
\end{rem}

Given $R,\epsilon>0$, we denote
\[
\Omega_{\epsilon,R}\coloneqq(\epsilon,R)\times (-\pi/4,\pi/4), \qquad \Omega'_{\epsilon,R}\coloneqq(\epsilon/2,2R)\times (-\pi/2,\pi/2)
\]
seen as subsets of $ Y_+=(0,\infty)\times S^1$. Using \cite{ams_BE_2016}*{Lemma 6.7}, let us fix a Lipschitz cut-off function $\chi_{\epsilon,R}\colon Y\to\R$ such that:
\begin{enumerate}[label=(\roman*),ref=\roman*]
  \item \label{it:cutoff-support}
  $0\le \chi_{\epsilon,R}\le 1$, with $\chi_{\epsilon,R}\equiv 1$ on $\Omega_{\epsilon,R}$
  and $\chi_{\epsilon,R}\equiv 0$ on $Y-\Omega'_{\epsilon,R}$.

  \item \label{it:cutoff-laplacian}
  $\Delta \chi_{\epsilon,R}\in L^{\infty}(Y)$.
\end{enumerate}
Furthermore, denoting $p\colon\R\to S^1$ the universal covering projection of $S^1$, we fix a smooth function $v_{}\colon S^1\to\R$ whose restriction to $(-3\pi/4,3\pi/4)\subset S^1$ is a section of $p$. Finally, we identify $v_{}$ with the smooth function $v_{}(r,\theta)\coloneqq v_{}(\theta)$ on $ Y_+=(0,\infty)\times S^1$. Observe that, thanks to our definition, we have
\begin{equation}\label{eq:gradient of v}
   \nabla_g v_{}=f^{-2}(r)\partial_\theta \text{ on }(0,\infty)\times(-3\pi/4,3\pi/4)\subset Y_+. 
\end{equation}
Note that $\chi_{\epsilon,R}\cdot v_{}$ is well-defined on $Y$ by our assumption on $\Spt(\chi_{\epsilon,R})$. Our second lemma establishes the regularity property of $\chi_{\epsilon,R}\cdot v_{}$.

\begin{lem}\label{lem: chi v in D(Delta)}
    The functions $\chi_{\epsilon,R}\cdot v_{}$ ($R,\epsilon>0$) are in $D(\Delta)$.
\end{lem}

\begin{proof}
    For convenience, let us write $\chi\coloneqq\chi_{\epsilon,R}$, $\Omega\coloneqq\Omega_{\epsilon,R}$, $\Omega'\coloneqq\Omega'_{\epsilon,R}$. Note that, thanks to Remark \ref{rem: v lip}, and since $\chi \in \Lip_c(Y)$ with $\Spt(\chi)\subset Y_+$, we have $\chi v\in \Lip_c( Y_+)$. Let us first show that $\chi v \in D(\mathbf{\Delta})$ (see \cite{gigli_nonsmooth}*{Definition 3.1.2}). We fix a test function $\phi\in\Lip_c(Y)$ and observe that, since $\chi v\equiv 0$ outside of $\Omega'$, then $\langle \nabla\phi,\nabla\chi v\rangle\equiv 0$ outside of $\Omega'$ thanks to \cite{gigli_nonsmooth}*{Theorem 2.2.3}. On the other hand, using the chain rule for $\nabla_g$ and the polarization identity for $\langle \nabla \phi,\nabla \chi v\rangle$, the following holds $\meas$-a.e.\ on $\Omega'$ thanks to Lemma \ref{lem: weak gradient vs lip}:
    \begin{equation}\label{eq: chain rule chi v}
        \begin{aligned}
            \langle \nabla \phi,\nabla (\chi v)\rangle
            &=\frac14\Big(|D(\phi+\chi v)|_w^2-|D(\phi-\chi v)|_w^2\Big)
              =\frac14\Big(\lip(\phi+\chi v)^2-\lip(\phi-\chi v)^2\Big)\\
            &=\frac14\Big(|\nabla_g(\phi+\chi v)|^2-|\nabla_g(\phi-\chi v)|^2\Big)
              =g(\nabla_g\phi,\nabla_g(\chi v))\\
            &=g(\nabla_g(v\phi),\nabla_g\chi)+g(\nabla_g(\chi\phi),\nabla_g v)-2\phi\, g(\nabla_g v,\nabla_g\chi).
        \end{aligned}
    \end{equation}
    Now, let us fix a smooth function $\psi\in\mathcal{C}^{\infty}_c( Y_+)$ such that $\psi\equiv 1$ on $\Omega'$. In particular, we have $g(\nabla_g(v\phi),\nabla_g\chi)=g(\nabla_g(v\phi\psi),\nabla_g\chi)$ (since $\chi\equiv 0$ outside of $\Omega'$) and $v\phi\psi\in\Lip_c(Y)$. Therefore, since $\chi\in D(\Delta)$, we have:
    \begin{equation}\label{eq: integration by part chi}
        \int_Yg(\nabla_g(v\phi),\nabla_g\chi)\di \mathfrak{m}
        =\int_Yg(\nabla_g(v\phi\psi),\nabla_g\chi)\di \mathfrak{m}
        =-\int_Yv\phi\psi\Delta\chi\di \mathfrak{m}
        =-\int_Yv\phi\Delta\chi\di \mathfrak{m},
    \end{equation}
    using the locality of the Laplacian for the last equality. Furthermore, using \eqref{eq:gradient of psi} and the fact that $\nabla_gv=f^{-2}(r)\partial_\theta$ on $\Omega'$ (see \eqref{eq:gradient of v}), we have:
    \begin{align*}
        \int_Yg(\nabla_g(\chi\phi),\nabla_g v)\di \mathfrak{m}
        &= \int_{r=0}^{\infty}\int_{\theta\in S^1}f^{-2}(r)\partial_\theta(\chi\phi)\rho(r)\di r\di \theta\\
        &=\int_{r=0}^{\infty}f^{-2}(r)\rho(r)\Bigg(\int_{\theta\in S^1}\partial_\theta(\chi\phi)\di \theta\Bigg)\di r\\
        &=0,
    \end{align*}
    where we observed that $g(\nabla_g(\chi\phi),\nabla_g v)\in L^{\infty}_{c}(Y)$ is integrable so that one can apply Fubini's theorem in the second equality and Stokes theorem for the third equality. Consequently, thanks to equations \eqref{eq: chain rule chi v} and \eqref{eq: integration by part chi}, we have:
    \begin{equation*}
            \int_Y\langle \nabla \phi,\nabla (\chi v)\rangle\di\mathfrak{m}
            =-\int_Y\phi(v\Delta\chi+2g(\nabla_g v,\nabla_g \chi))\di\mathfrak{m},
    \end{equation*}
    where, thanks to item \eqref{it:cutoff-laplacian}, we have $v\Delta\chi+2g(\nabla_g v,\nabla_g \chi)\in L^{\infty}_c(Y)$. Therefore, we have $\chi v\in D(\mathbf{\Delta})$, and $\mathbf{\Delta}\chi v = (v\Delta\chi+2g(\nabla_g v,\nabla_g \chi))\meas$. In particular, thanks to \cite{gigli_diff_2015}*{Proposition 4.24}, we have $\chi v\in D(\Delta)$ and $\Delta \chi v = v\Delta\chi+2g(\nabla_g v,\nabla_g \chi)$.
\end{proof}

\begin{rem}
    We used the chain rule for $\nabla_g$ and the polarization identity for $\langle \nabla \phi,\nabla \chi v\rangle$ to obtain equation \eqref{eq: chain rule chi v} because, as stated in \cite{gigli_nonsmooth}, the chain rule for $\nabla$ a priori only holds for products of functions in $W^{1,2}(Y)$. However, the local Lipschitz constant of $v$ may a priori go to infinity near $\{r=0\}$.
\end{rem}

\begin{rem}
    We introduced a smooth function $\psi\in\mathcal{C}^{\infty}_c( Y_+)$ such that $\psi\equiv 1$ on $\Omega'$ so that we were able to use the integration by parts w.r.t.\ $\chi$ in equation \eqref{eq: integration by part chi}, which a priori only works when paired with a $W^{1,2}(Y)$ function. While we do have $v\phi\psi\in\Lip_c(Y)\subset W^{1,2}(Y)$, we may not have $v\phi\in W^{1,2}(Y)$.
\end{rem}

We now have all we need to prove Proposition \ref{prop: warping function is sobolev}.

\begin{proof}[Proof of Proposition \ref{prop: warping function is sobolev}]
    Let us fix $0<\epsilon<R$ and introduce $\chi_{\epsilon,R}$ satisfying items \eqref{it:cutoff-support} and \eqref{it:cutoff-laplacian}. As before, we denote $\chi\coloneqq\chi_{\epsilon,R}$, $\Omega\coloneqq\Omega_{\epsilon,R}$, $\Omega'\coloneqq\Omega'_{\epsilon,R}$ for conciseness. Observe that, since $v\in\Lip_{loc}( Y_+)$ and since $\chi \in\Lip_c(Y)$ with $\Spt(\chi)\subset Y_+$, we have $\chi v\in\Lip_{c}( Y_+)$. As a result, $|D\chi v|_w = \lip(\chi v)\in L^{\infty}(Y)$, i.e.\ $\chi v$ has bounded gradient. Moreover, thanks to Lemma \ref{lem: chi v in D(Delta)} and \cite{gigli_nonsmooth}*{Proposition 3.3.18}, we have $\chi v\in H^{2,2}(Y)$. Therefore, thanks to \cite{gigli_nonsmooth}*{Proposition 3.3.22}, we have $\langle\nabla \chi v,\nabla \chi v\rangle\in W^{1,2}(Y)$. Using the polarization identity for $\langle\nabla \chi v,\nabla \chi v\rangle$, the fact that $\chi\equiv 1$ on $\Omega$, and Lemma \ref{lem: weak gradient vs lip}, the following holds on $\Omega$:
    \begin{equation}\label{eq: f-2}
        \langle\nabla \chi v,\nabla \chi v\rangle =\lip(\chi v)^2= \lip(v)^2 = |\nabla_gv|^2=f^{-2},
    \end{equation}
    where the last equality follows from \eqref{eq:gradient of v}.
    Since $f$ is continuous and positive, there exists $C>0$ such that $C^{-1}\le f^{-2}\le C$ on $\Omega$. Let us fix $\phi\in\mathcal{C}_c^{\infty}(\R)$ such that $\phi(t)=t^{-1/2}$ for $t\in [C^{-1},C]$ and $\phi(0)=0$. Thanks to the chain rule \cite{gigli_nonsmooth}*{Corollary 2.2.8}, we have
    \[
    \psi\coloneqq\phi\circ \langle\nabla \chi v,\nabla \chi v\rangle\in W^{1,2}(Y)\text{ with compact support }\Spt(\psi)\subset  Y_+.
    \]
    Moreover, thanks to equation \eqref{eq: f-2}, $\psi= f$ on $\Omega$. As a result of \cite{MondinoRyborz}*{Corollary 4.27}, we have $|D\psi|_{w}=|\nabla_g\psi|\in L^2( Y_+)$, where $\nabla_g\psi$ is defined in the weak sense as the weak gradient of $\psi$. Since $g=\di r^2+f^2(r)\di\theta^2$ and $\psi(r,\theta)=f(r)$ on $\Omega$, we have $\nabla_g\psi(r,\theta)=f'(r)\in L^2(\Omega)$, i.e.
    \[
    \int_\epsilon^R\int_{\theta\in S^1}|f'(r)|^2\rho(r)\di r\di\theta=2\pi\int_\epsilon^R|f'(r)|^2\rho(r)\di r<\infty.
    \]
    In particular, since $\rho$ is continuous on $(0,\infty)$ as a $\CDe(0,N)$ density and since $\m$ has full support, there exists $c>0$ such that $\rho\ge c$ on $[\epsilon,R]$, which implies that $f'\in L^2([\epsilon,R])$. Since this holds for arbitrary $0<\epsilon<R$, we conclude that $f\in W^{1,2}_{loc}(0,\infty)$.
\end{proof}

\begin{rem}
    Strictly speaking, \cite{MondinoRyborz} deals with the case where $(M,g)$ is complete. However, their approach is local by nature and all of their arguments carry on verbatim to our situation. In particular, in the proof above, we equated $f$ with a Sobolev function $\psi\in W^{1,2}(Y)$ with compact support $\Spt(\psi)\subset Y_+$. Thanks to \cite{Ambrosio-Gigli-Savare_14}*{(2.22)}, there exists a sequence of Lipschitz functions $\psi_n$ converging to $\psi$ in $L^2(Y)$ such that $\lip(\psi_n)$ converges to $|D\psi|_w$ in $L^2(Y)$. Multiplying $\psi_n$'s by a cut-off function equal to $1$ on a neighborhood of $\Spt(\psi)$ and with compact support in $ Y_+$, we may assume without loss of generality that $\psi_n\in \mathrm{Lip}_c( Y_+)$. From there, the arguments of the proof of \cite{MondinoRyborz}*{Lemma 4.26} apply and we may conclude that $|D\psi|_w=|\nabla_g\psi|$ holds almost everywhere.
\end{rem}

\begin{rem}
    Thanks to the previous proof, $f$ can be seen as a locally Sobolev function on $(0,\infty)$ equipped with the Lebesgue measure, but also as a locally Sobolev function on $(Y_+,d_g,\meas\llcorner Y_+)$.
\end{rem}

\subsection{From \texorpdfstring{$\RCD$}{RCD} to distributional Ricci curvature lower bound}\label{subsec:RCD_to_BE}
The goal of this section is to prove the distributional Bakry--Em\'ery Ricci curvature defined on $( Y_+,d_{| Y_+},\meas{\llcorner  Y_+}) = ((0,\infty)\times S^1, d_g,\rho(r) \di r\di \theta)$ has lower bound $0$ and derive useful relations between the warping function $f$ and the density $\rho$, which is the content of Proposition \ref{prop: measure-inequality}. 

Recall that from Proposition \ref{prop: warping function is sobolev} we have $f\in C^0( Y_+)\cap W^{1,2}_{loc}( Y_+)$ is positive, and from Lemma \ref{lem:measure density} that $\rho$ is positive non-decreasing and $\frac1N$-concave. In particular,  $\rho\in C^0( Y_+)\cap W^{1,2}_{loc}( Y_+)$. We are therefore in the setting of Mondino--Ryborz \cite{MondinoRyborz}. A subtle difference is that our $\RCD(0,N)$ space $(Y,d,\meas)$ has a boundary, and \cite{MondinoRyborz} considers manifolds without boundary. Nevertheless, we only consider distributions on the interior $( Y_+,d_{| Y_+})$. We will see in the following that the arguments of \cite{MondinoRyborz} remain applicable.

Notice that the smooth structure on $ Y_+$ is the standard one, and $ Y_+$ admits the global chart $(r,\theta)$. We can define and compute distributional derivatives in the global chart. Let $\partial_r$, $\partial_\theta$ be the smooth tangent vectors of $r$, $\theta$ coordinates respectively. The Christoffel symbols of $g=\di r^2+ f(r)^2 \di \theta^2$ with respect to $\partial_r$, $\partial_\theta$ are computed as follows.
$$\Gamma^r_{\theta\theta}=-ff',\quad \Gamma^\theta_{r\theta}=\Gamma^\theta_{\theta r}=\dfrac{f'}{f}. $$
All others are zero. 
So all Christoffel symbols are in $L^2_{loc}( Y_+)$ since $f$ and $1/f$ are locally bounded and $f'\in L^2_{loc}( Y_+)$. Following \cite{MondinoRyborz}*{Section 2}, the Bakry--Emery Ricci tensor is defined as 
\[
\Ric_{\m,\infty}\coloneqq\Ric+\hess\,[\ln f-\ln\rho],
\]
where the Ricci tensor and Hessian are defined via distributional derivatives in the global chart $(r,\theta)$ as follows:
\begin{align}
     \hess\, u(\partial_j,\partial_k)&\defeq \partial_{j}\partial_k u-\Gamma^s_{jk}\partial_s u,\quad u\in W^{1,2}_{loc}( Y_+)\cap C^0( Y_+).\\
    \Ric(\partial_j,\partial_k)&\defeq \partial_p \Gamma^p_{jk}-\partial_j\Gamma^p_{pk}+\Gamma^s_{kj}\Gamma^p_{ps}-\Gamma^s_{kp}\Gamma^p_{js}, \quad j,k,p,s\in \{r,\theta \};
\end{align}
see \cite{MondinoRyborz}*{Definition 3.7}. Let $b\defeq f'/f=(\ln f)'\in L^2_{loc}( Y_+)$. The only two non-trivial components of the Ricci tensor can be expressed as 
\begin{align}
    \Ric_{\m,\infty}(\partial_r,\partial_r)&= -b'-b^2+ b'-(\rho'/\rho)'=-b^2-\left(\frac{\rho'}{\rho}\right)'\label{eq:Ricrr}\\
    \Ric_{\m,\infty}(\partial_\theta,\partial_\theta)&={(-ff')'}+2b^2f^2-\frac{bf^2\rho'}{\rho}=-ff''-(f')^2 +2b^2f^2-\frac{bf^2\rho'}{\rho}.\label{eq:Ricvv}
\end{align}

\begin{rem}
    We justify the chain rule for distributions $(ff')'=ff''+(f')^2$ as follows. Take a test function $\phi\in C_c^\infty(0,\infty)$, and a cut-off function $\chi$ such that $0\le \chi\le 1$, $\chi$ is $1$ in $\Spt(\phi)$ and $\chi$ is compactly supported. We see that $\chi f\in W^{1,2}(0,\infty)$ and $\chi'=\chi''=0$ in $\Spt(\phi)$. Then we have that $\chi f\in W^{1,2}(0,\infty)$ hence, $(\chi f)''\in W^{-1,2}(0,\infty)$. We have a well-defined paring through the duality $W^{-1,2}(0,\infty)=(W^{1,2})^*(0,\infty)$:
    \[
    \langle ff'',\phi \rangle\defeq \langle \chi f'',\chi f\phi \rangle=\langle (\chi f)'',\chi f\phi \rangle.
    \]
    It is clear that this does not depend on the choice of $\chi$. With this understanding, we can compute that 
    \begin{align*}
        \langle ff'',\phi \rangle&=  \langle f'',f \phi \rangle=\int_{0}^{\infty} - f'(f\phi)' \\
       &=   -\int_{0}^{\infty} (f')^2\phi-  \int_{0}^{\infty} ff'\phi' =-\langle (f')^2,\phi \rangle+ \langle (ff')',\phi \rangle,
    \end{align*}
    which is the claimed chain rule for $(ff')'$.
\end{rem}



Following Mondino--Ryborz \cite{MondinoRyborz}*{Section 5-6}, we show that if the entire space $(Y,d,\meas)$ satisfies the $\RCD(0,N)$ condition (hence the $\RCD(0,\infty)$ condition), then the distributional Ricci curvature on the regular set $ Y_+$ has lower bound $0$. In $\RCD$ spaces, we will need test objects with Sobolev regularity in place of the smooth test objects for distribution theory. We define the space of test functions 
\[
\mathrm{TestF}(Y)\defeq \left\{f\in D(\Delta)\cap L^\infty(Y):|\nabla f|\in L^\infty (Y), \Delta f\in W^{1,2}(Y)\right\},
\]
and the space of test vectors
\[
\mathrm{TestV}(Y)\defeq \left\{\sum_{i=1}^n h_i\nabla f_i:n\in \N, f_i,h_i\in\mathrm{TestF}(Y), i=1,2,\ldots, n\right\},
\]
as in \cite{gigli_nonsmooth}*{Chapter 3}.

\begin{rem}\label{rem:product_of_cutoff_and_testV}
    The set of test functions is an algebra \cite{gigli_nonsmooth}*{(3.1.8)} and there exist good cut-off functions which are test functions (see \cite{gigli_nonsmooth}*{3.3.25} after \cite{ams_BE_2016}*{Lemma 6.7}), so the product of good cut-off function with a test vector field is still a test vector field.
\end{rem}

The $\RCD$ condition and distributional Ricci curvature are bridged through the measure valued Ricci curvature $\mathbf{Ric}$, defined in \cite{gigli_nonsmooth}*{Theorem 3.6.7 (3.6.10)}, inspired by the Bochner formula. The full definition of $\mathbf{Ric}$ requires the introduction of functional spaces and notations that will not be used elsewhere, so we refer the reader to \cite{gigli_nonsmooth}*{Section 3.6} for the details. Next, we record that the $\RCD$ condition implies the lower bound of $\mathbf{Ric}$. 
\begin{thm}[\cite{gigli_nonsmooth}*{Theorem 3.6.7 (3.6.11)}]\label{thm:RCD-to-RicciMeasure}
       For any $U\in \mathrm{TestV}(Y)$, the measure valued Ricci tensor satisfies $\mathbf{Ric}(U,U)\ge 0$.
\end{thm}

Furthermore, the measure valued Ricci curvature coincides with the distributional Ricci curvature when testing on test vector fields by \cite{MondinoRyborz}*{Proposition 5.21}.
\begin{prop}\label{prop:distribution-is-measure}
For any $\phi\in C_c^\infty( Y_+)$ and $U\in \mathrm{TestV}(Y)$, it holds
\[
\int_Y \phi\di \mathbf{Ric}(U,U)=\int_{ Y_+} \phi\Ric_{\m,\infty}(U,U)  \rho\di r\di v\defeq \langle\Ric_{\m,\infty}(U,U), \phi\rangle.
\]
\begin{proof}
     Let $\chi$ be a good cut-off function which is a test function on $ Y_+$ such that $0\le \chi\le 1$, $\chi=1$ in $\Spt(\phi)$ and $\chi$ compactly supported in $ Y_+$. We see that $\chi U\in \mathrm{TestV}(Y)$ thanks to Remark \ref{rem:product_of_cutoff_and_testV} and is supported in $ Y_+$. Then the result follows from the local computation in the proof of \cite{MondinoRyborz}*{Proposition 5.21}.
\end{proof}
\end{prop}

Combining Theorem \ref{thm:RCD-to-RicciMeasure} and Proposition \ref{prop:distribution-is-measure}, we obtain the following result. 

  \begin{thm}\label{lem:BERic-bound}
      For any smooth vector field $U\in \Gamma(T Y_+)$ and any smooth function $\phi\in C_c^\infty( Y_+)$, $\phi\ge 0$, we have $\langle\Ric_{\m,\infty}(U,U),\phi\rangle\ge 0$, which implies $\Ric_{\m,\infty}(U,U)\ge 0$, as a scalar distribution on $ Y_+$. In particular 
  \begin{equation}\label{eq:distributional-Ricci-nonnegative}
      -b^2-\left(\frac{\rho'}{\rho}\right)'=\Ric_{\m,\infty}(\partial_r,\partial_r)\ge 0, \quad -\frac{f^2}{\rho}(b\rho)'=\Ric_{\m,\infty}(\partial_\theta,\partial_\theta)\ge 0 
  \end{equation}
  as distributions.
  \end{thm}
  
  The proof goes verbatim as \cite{MondinoRyborz}*{Theorem 6.11} after standard cutoff procedure. We sketch the ideas without giving computation details and refer the readers to the original proof.
  
\begin{proof}
    Let $\chi$ be a smooth cut-off function on $ Y_+$ such that $0\le \chi\le 1$, $\chi=1$ in $\Spt(\phi)$ and $\chi$ compactly supported in $ Y_+$. Then by \cite{MondinoRyborz}*{Lemma 6.10}, $\chi U$ can be approximated by test vectors in $W^{1,2}$ norm. Let $V\in \mathrm{TestV}(Y)$, we have by Proposition \ref{prop:distribution-is-measure} and Theorem \ref{thm:RCD-to-RicciMeasure} that
    \begin{align*}
        \langle\Ric_{\m,\infty}(U,U), \phi\rangle&=\langle\Ric_{\m,\infty}(\chi U,\chi U), \phi\rangle\\
        &= \langle\Ric_{\m,\infty}(\chi U,\chi U), \phi\rangle-\langle\Ric_{\m,\infty}(V,V),\phi\rangle+ \int_Y \phi\di \mathbf{Ric}(V,V)\\
        &\ge \langle\Ric_{\m,\infty}(\chi U,\chi U), \phi\rangle-\langle\Ric_{\m,\infty}(V,V),\phi\rangle.
    \end{align*}
    It remains to show that if there is a  $V_i$ such that $V_i\to \chi U$ in $W^{1,2}(T Y_+)$ then  
    \[ \langle\Ric_{\m,\infty}(V_i,V_i),\phi\rangle\to \langle\Ric_{\m,\infty}(\chi U,\chi U), \phi\rangle.\]
     This is indeed the case as shown in the proof of \cite{MondinoRyborz}*{Theorem 6.11}. 
     
     The last assertion follows by taking $U$ to be $\partial_r$ and $\partial_v$.
\end{proof}

In general, one can only multiply a distribution by functions of the class $\C_c^\infty$. In our setting, the distributional Ricci curvature lower bound provides more regularity so that we have more freedom to rewrite \eqref{eq:distributional-Ricci-nonnegative}. 

  \begin{prop}\label{prop: measure-inequality}
       We have the following inequalities as Radon measures:
      \[-(\rho b)'\ge 0, \quad  -\rho''\rho+(\rho')^2\ge \rho^2b^2.\]
  \end{prop}
\begin{proof}
            Notice that $b=f'/f$, we can simplify \eqref{eq:Ricvv} to be
\[
 \Ric_{\m,\infty}(\partial_\theta,\partial_\theta)={(-ff')'}+2b^2f^2-\frac{bf^2\rho'}{\rho}=-\frac{f^2}{\rho}(b\rho)'.
\]
  By Theorem \ref{lem:BERic-bound}, we have that
  $$
  -b^2-\left(\frac{\rho'}{\rho}\right)'=\Ric_{\m,\infty}(\partial_r,\partial_r)\ge 0, \quad -\frac{f^2}{\rho}(b\rho)'=\Ric_{\m,\infty}(\partial_\theta,\partial_\theta)\ge 0
  $$ 
  as distributions. It then follows from Riesz representation theorem that $\Ric_{\m,\infty}(\partial_r,\partial_r)$ and $\Ric_{\m,\infty}(\partial_v,\partial_v)$ are Radon measures. We can multiply both inequalities by positive continuous functions to get
  \[
  -(b\rho)'\ge 0, \quad -\rho^2\left(\frac{\rho'}{\rho}\right)'\ge \rho^2b^2.
  \]
  So it remains to simplify $\left(\frac{\rho'}{\rho}\right)'$. We observe that, if we can show that the distribution $\rho''$ is in fact a Radon measure, then we have the chain rule for $BV$ functions at our disposal. Applying the chain rule yields that $\left(\frac{\rho'}{\rho}\right)'=\frac{\rho''\rho-(\rho')^2}{\rho^2}$ holds as measures. 
  
  It remains to show that $\rho''$ is a Radon measure. Let $u=\rho^{1/N}$, then, by concavity, $u''\le 0$ as a distribution. By the Riesz representation theorem, $u''$ is a Radon measure. Also, $u$ is concave so it is locally Lipschitz, hence $u\in W_{loc}^{1,\infty}(0,\infty)$. Since $t \to t^N$ is in $C^1(0,\infty)$, $\rho$ is also in $W_{loc}^{1,\infty}(0,\infty)$. Then $ u' = \frac1N \rho' \rho^{\frac1N-1}$ holds pointwise with canonical representatives, i.e. the left or right derivatives of $u$ and $\rho$. Meanwhile, $u'$ is in $BV_{loc}(0,\infty)$. Moreover, we have $ \rho^{1-\frac1N}$ is in $W^{1,\infty}_{loc}(0,\infty)$. Therefore, $\rho' = N u' \rho^{1-\frac1N}$ is also in $BV_{loc}(0,\infty)$ and, by the $BV-W^{1,\infty}$ product rule, we have $\rho''=N \rho^{1-\frac1N} u '' +(N-1) u' \rho^{-\frac1N} \rho'$ as a signed measure. 
\end{proof}


\section{Rigidity results}\label{sec:rigidity}
\subsection{Statements and remarks on the rigidity results}

As mentioned in the introduction, the plane/halfplane rigidity (Theorem \ref{thm:rigid_p/h}) shall follow from two rigidity results, the sharp halfplane rigidity and the sharp ray rigidity, concerning free $\R$-action and isotropic $S^1$-action separately. We restate these two rigidity results in this section for readers' convenience and give some related examples and remarks. Then we prove Theorem \ref{thm:rigid_p/h} by assuming the sharp halfplane rigidity and the sharp ray rigidity.

\begin{thm}[Sharp Halfplane rigidity]\label{thm:hp_rigid_vol<2}
   Let $(Y,y,d,\mathfrak{m})$ be an $\RCD(0,N)$. Suppose that\\
   (1) $Y$ has an isomorphic $G$-action, where $G\simeq \mathbb{R}$ is closed in $\mathrm{Isom}(Y)$, such that the quotient metric space $(Y/G,\bar{y})$ is isometric to a ray $([0,\infty),0)$;\\
   (2) the measure $\meas$ satisfies $$\lim_{r\to\infty}\dfrac{\meas (\Omega_r)}{r^2}=0.$$
   Then the metric space $(Y,d)$ is isometric to a Euclidean halfplane.
\end{thm}

\begin{exmp}\label{exmp:hp_2_sharp}
   The degree $2$ in the assumption (2) of Theorem \ref{thm:hp_rigid_vol<2} is sharp. In fact, we have the following example. Let $Y_+$ be the open halfplane $\{(r,v)\mid r>0,v\in \R\}$. We define a weighted smooth Riemannian metric on $Y_+$ by
   $$g=\di r^2 + r^{-2\alpha} \di v^2,  \quad \meas = r \di r \di v,$$
   where $\alpha>0$. The measure $\meas$ has growth $\meas(\Omega_r)=r^2$. Denoting $e_1=\partial_r$ and $e_2=r^\alpha \partial_v$, we compute that the space $(Y_+,g,\meas)$ has $N$-Barky-\'Emery Ricci curvature
   $$\Ric_\meas^N(e_1,e_1)=1-\alpha^2 -\dfrac{(1+\alpha)^2}{N-2},\quad  \Ric_\meas^N(e_2,e_2)=0.$$
   It is clear that we can choose suitable $\alpha$ and $N$ such that $\Ric_\meas^N\ge 0$; for example, $\alpha=1/3$ and $N=4$. Let $Y$ be the metric completion of $(Y_+,d)$, which is homeomorphic to a closed halfplane. We extend $\meas$ to $Y$ by setting $\meas(\{r=0\})=0$. Then $(Y,(0,0),d,\meas)$ is an $\RCD(0,N)$ space (see, for example, the argument in \cite{RS23}*{Section 3.5}) that satisfies assumption (1) of Theorem \ref{thm:hp_rigid_vol<2} and $\meas(\Omega_r)=r^2$, but $(Y,d)$ is not isometric to a Euclidean halfplane.
\end{exmp}


\begin{thm}[Sharp ray rigidity]\label{thm:ray_rigid}
   Let $(Y,y,d,\mathfrak{m})$ be an $\RCD(0,N)$ space. Suppose that\\
   (1) $Y$ has an isomorphic $S^1$-action, where $S^1$ fixes $y$, such that the quotient metric space $(Y/S^1,\bar{y})$ is isometric to a ray $([0,\infty),0)$;\\
   (2) the measure $\meas$ satisfies 
   $$\lim_{r\to 0^+}\dfrac{\meas (B_r(y))}{r^2}=+\infty.$$
   Then the metric space $(Y,d)$ is isometric to a ray and the $S^1$-action on $Y$ is trivial.
\end{thm}

\begin{rem}\label{rem:ray_2_sharp}
   Similar to the sharp halfplane rigidity (Theorem \ref{thm:hp_rigid_vol<2}), the degree $2$ in Theorem \ref{thm:ray_rigid} is also sharp. In fact, the Euclidean plane $\R^2$ with the Lebesgue measure provides a counter-example with
   $$\lim_{r\to 0^+}\dfrac{\meas (B_r(y))}{r^2}<+\infty.$$
\end{rem}

Assuming Theorem \ref{thm:ray_rigid}, we can deduce a corresponding ray/line rigidity result for any general abelian fixed-point action as below. This generalizes \cite{NPZ}*{Proposition 4.1}, where $\meas(B_r(y))$ is assumed to be linear.

\begin{cor}\label{cor:rigid_rayline}
  Let $(Y,y,d,\mathfrak{m})$ be an $\RCD(0,N)$ space. Suppose that\\
  (1) $\mathrm{Isom}(Y)$ has an abelian Lie subgroup $K$ fixing $y$ such that the quotient metric space $(Y/K,\bar{y})$ is isometric to a ray $([0,\infty),0)$;\\
  (2) the measure $\meas$ satisfies $$\lim_{r\to 0^+}\dfrac{\meas (B_r(y))}{r^2}=+\infty.$$
   Then the metric space $(Y,d)$ is isometric to a ray or a line; as a consequence, $K=\{e\}$ or $\Z_2$.
\end{cor}

\begin{proof}[Proof assuming Theorem \ref{thm:ray_rigid}]
   We first assume that $K$ is connected and show that $K=\{e\}$. Suppose otherwise, then $K$ is a torus $T^l$, where $l\ge 1$, and acts effectively and isomorphically on $Y$. Let $H\leq T^l$ be a closed subgroup isomorphic to $T^{l-1}$. Then $T^l/H\simeq S^1$ acts effectively and isomorphically on $Y/H$. Note that the quotient space $(Y/H,\bar{y},\bar{\meas})$ is $\RCD(0,N)$ with a $S^1$-action that fulfills the condition of Theorem \ref{thm:ray_rigid}. Hence $Y/H$ is isometric to a ray thanks to Theorem \ref{thm:ray_rigid}; a contradiction to the effective $S^1$-action on $Y/H$.

   For a general abelian Lie subgroup $K$, because $K$ fixes $y$, it is compact. We can write $K=T^l\times F$, where $F$ is a finite group. Applying the first paragraph to the quotient space $Y/F$, we obtain that $Y/F$ is isometric to a ray and $K=\{e\}$. It follows that $Y$ has rectifiable dimension $1$. By \cite{Kitabeppu-Lakzian_16}, $Y$ is isometric to a line or a ray.
\end{proof}

Now we prove plane/halfplane rigidity (Theorem \ref{thm:rigid_p/h}) from Theorem \ref{thm:hp_rigid_vol<2} and Corollary \ref{cor:rigid_rayline}.

\begin{proof}[Proof of Theorem \ref{thm:rigid_p/h} assuming Theorems \ref{thm:hp_rigid_vol<2} and \ref{thm:ray_rigid}]
   We first consider the quotient metric measure space of $Y$ by the compact group $K$:
   $$ (Y',y',d',\meas')= (Y,y,d,\meas)/K.$$
   It has an isometric and measure preserving $G'\simeq G/K\simeq \R$-action. Hence $(Y',y',d',\meas')$ is an $\RCD(0,N)$ space and satisfies condition (1) in Theorem \ref{thm:hp_rigid_vol<2}. Under the quotient map $\pi:Y\to Y'$, the strip region $\Omega'_r\subseteq Y'$ has a pre-image $\pi^{-1}(\Omega'_r)=\Omega_r\subseteq Y$. Hence
   $$\meas'(\Omega'_r)=\meas(\pi^{-1}(\Omega'_r))=\meas(\Omega_r).$$
   Together with the given $\meas$-growth condition
   $$ \lim_{r\to\infty} \dfrac{\meas'(\Omega'_r)}{r^2}=\lim_{r\to\infty} \dfrac{\meas(\Omega_r)}{r^2}=0,$$
   we see that $(Y',y',d',\meas')$ satisfies all the assumptions in Theorem \ref{thm:hp_rigid_vol<2}. Therefore, $Y'$ is isometric to the Euclidean halfplane $\R\times [0,\infty)$. 
   
   Next, we lift the line in $Y'=Y/K$ to $Y$ and apply Gigli's splitting theorem \cite{Gigli13}, then $(Y,y,d,\meas)$ is isomorphic to $(\R,0,d_E,\mathcal{L})\otimes (Z,z,d_Z,\meas_Z)$ for some $\RCD(0,N-1)$ space $(Z,z,d_Z,\meas_Z)$. $K$ acts trivially on the $\R$-factor of $Y$. Now we have $Kz=z$ and the quotient $Z/K$ isometric to a ray $[0,\infty)$. Moreover, under the splitting $Y=\R\times Z$, the region $\Omega_r\subseteq Y$ is identified with $[-1,1]\times B_r(z)$. Hence
   $$\lim_{r\to 0^+} \dfrac{\meas_Z (B_r(z))}{r^2} =\lim_{r\to 0^+} \dfrac{\meas(\Omega_r)}{2r^2} = \infty.$$
   Thanks to Corollary \ref{cor:rigid_rayline}, $Z$ is isometric to a line or a ray. We conclude that $Y=\R\times Z$ is isometric to a Euclidean halfplane or a Euclidean plane.
\end{proof}

\begin{rem}\label{rem:p/h_measure}
  In Theorem \ref{thm:rigid_p/h}, when $Y$ is isometric to a Euclidean plane, the measure $\meas$ must be (a multiple of) the Lebesgue measure thanks to Gigli's splitting theorem \cite{Gigli13}. When $Y$ is isometric to a Euclidean halfplane $\R\times [0,\infty)$, again by Gigli's splitting theorem, the measure $\meas$ splits as $\mathcal{L}\otimes \underline{\meas}$, where $\mathcal{L}$ denotes the Lebesgue measure on $\R$ and $\underline{\meas}$ is some measure on $[0,\infty)$. In general, the measure $\underline{\meas}$ is distinct from the Lebesgue measure. Nevertheless, because $$\Omega_r=[-1,1]\times [0,r]\subseteq Y=\R\times [0,\infty),$$
  it is clear that $\underline{\meas}$ satisfies
 $$\lim_{r\to 0^+} \dfrac{\underline{\meas}([0,r])}{r^2}=+\infty,\quad \lim_{r\to\infty} \dfrac{\underline{\meas}([0,r])}{r^2}=0.$$
\end{rem}

\begin{rem}\label{rem:p/h_action}
   Inspecting the proof of Theorem \ref{thm:rigid_p/h} above, we have $Z/K=[0,\infty)$ and $Z$ is either a ray or a line. Hence the isotropy subgroup $K$ has only two possibilities, either trivial or $\Z_2$. As a consequence, the $G$-space $(Y,y,d,G)$ under the assumption of Theorem \ref{thm:rigid_p/h} is isomorphic to one of the following:\\
   (1) Euclidean halfplane $(\R\times [0,\infty),0,\R)$,\\
   (2) Euclidean plane $(\R^2,0,\R\times \Z_2)$, where $\Z_2$ acts as reflection with respect to the line $\R\cdot 0$.
\end{rem}

\subsection{Proof of sharp halfplane rigidity}

In this section
\[
(Y,y,d,\meas)\text{ denotes an } \RCD(0,N) \text{ space satisfying the hypotheses of Theorem \ref{thm:hp_rigid_vol<2}}.
\]
Based on the distributional Ricci curvature lower bound derived in Proposition \ref{prop: measure-inequality}, we will show that $(Y,d)$ is isometric to a halfplane.

First of all, we show that taking the quotient of $Y$ by its $\mathbb{Z}$-action leads to an $\RCD(0,N)$ space with an $S^1$-action whose quotient is a ray.

\begin{lem}\label{lem:half_plane_hypothesis_to_cylinder}
    There exists a unique measure ${\meas}'$ on $Y'\coloneqq Y/\mathbb{Z}$ such that the quotient map $\pi\colon Y\to Y'$ induces a local isomorphism of m.m.s.\@ from $(Y,d,\meas)$ to $(Y',{d}',{\meas}')$, where ${d}'$ is the quotient distance. In particular, $(Y',{y}',{d}',{\meas}')$ satisfies the hypothesis (\ref{cond:S1_action_and_ray_quotient}) and we have:
    \[
    \lim_{r\to\infty}\frac{\meas'(\Omega'_r)}{r^2}=0,
    \]
    where ${y}'=\pi(y)$ and $\Omega'_r\coloneqq\pi(\Omega_r)$.
\end{lem}

\begin{proof}
The quotient map $\pi$ is a covering projection, so we may define a unique nonnegative Radon measure $\meas'$ on $Y'$ such that, given any $U\subset Y$ such that $\pi_{\lvert U}\colon U\to \pi(U)$ is a homeomorphism, we have $\meas'(\pi(U))=\meas(U)$. It is then clear that $\pi$ induces a local isomorphism of m.m.s.\@ from $(Y,d,\meas)$ to $(Y',{d}',{\meas}')$; hence, by the locality of the $\RCD(0,N)$ condition (see \cite{SZ23}*{Theorem 18} after \cite{erbar_equivalence_2015}*{Theorem 3.25}), $(Y',{d}',{\meas}')$ is an $\RCD(0,N)$ space. Then, since $\R$ is abelian, $\R/\mathbb{Z}=S^1$ acts on $Y'=Y/\mathbb{Z}$ and it is clear that the action is by isomorphisms of metric measure spaces and that the quotient is a ray. Therefore, $(Y',{y}',{d}',{\meas}')$ satisfies the hypothesis (\ref{cond:S1_action_and_ray_quotient}).

We denote $Y_+$ and $Y'_+$ the pre-images of $(0,\infty)$ by the quotient maps $Y\to Y/\R=[0,\infty)$ and $Y'\to Y'/S^1=[0,\infty)$. Proceeding as for \ref{lem:free_at_reg}, we can identify $Y_+$ and $Y'_+$ as $(0,\infty)\times\R$ and $(0,\infty)\times S^1$, equipped with their respective standard $\R$-action and $S^1$-action. Thanks to Remark \ref{rem:orbits_have_measure_zero}, which asserts that horizontal rays and $S^1$-orbits have measure zero in $Y'$, we have $\m'(\Omega'_r)=\m'((0,r)\times(-\pi,\pi))$.  Since $\pi\colon Y\to Y'$ is a local isomorphism of m.m.s.\@, the same remark holds for $Y$ and we have:
\[
\m(\Omega_r)=\m((0,r)\times(-1,1))=\m((0,r)\times(-1,0))+\m((0,r)\times(0,1))=2\m((0,r)\times(0,1)).
\]
However, $\pi\colon(0,r)\times(0,1)\to(0,r)\times (-\pi,\pi)$ is a homeomorphism; thus $\meas(\Omega_r)=2\meas'(\Omega'_r)$. The conclusion follows since $\lim_{r\to\infty}r^{-2}\meas(\Omega_r)=0$.
\end{proof}

We continue to use the notation introduced in the proof of Lemma \ref{lem:half_plane_hypothesis_to_cylinder}. Recall that $Y'_+$ is the pre-image of $(0,\infty)$ by the quotient map $Y'\to Y'/S^1=[0,\infty)$. In the previous section, concluding with Proposition \ref{prop: measure-inequality}, we identified $(Y'_+,d_{\lvert Y'_+},\meas\llcorner Y'_+)$ with the open half-cylinder $(0,\infty)\times S^1$ equipped with a Riemannian metric $g=\di r^2+f^2(r)\di \theta^2$ and a weighted measure $\rho(r)\di r \di \theta$, where $f\in W^{1,2}_{loc}(0,\infty)$ is positive and continuous, $\rho$ is a positive $\CDe(0,N)$ density on $(0,\infty)$, and the following inequalities hold as Radon measures:
\[
-(\rho b)'\ge0,\quad -\rho''\rho+(\rho')^2\ge \rho^2b^2,\quad\text{where }b=\ln(f)'.
\]

We start with a simple lemma that converts our volume growth condition $\lim_{r\to\infty}\frac{\meas(\Omega'_r)}{r^2}=0$ (obtained in Lemma \ref{lem:half_plane_hypothesis_to_cylinder}) into a growth condition of the density $\rho$. 

\begin{lem}\label{rem:volumegrowth-to-rho}
   $\lim_{r\to\infty} \rho(r)/r=0$.
\end{lem}

\begin{proof}
   Notice that
\[
\meas'(\Omega'_r)=\int_0^r\int_{S^1} \rho(s)\di s\di \theta=2\pi\int_{0}^r \rho(s)\di s,
\]
and $\rho$ is non-decreasing since it is a positive $\CDe(0,N)$ density on $(0,\infty)$. As $r\to\infty$, we deduce that
\begin{equation*}
    \frac{\rho(r/2)}{r/2}\le 4r^{-2} {\int_{r/2}^{r}\rho(s) \di s} \le \frac{2}{\pi}\dfrac{\meas'(\Omega'_r)}{r^2}\to 0.
\end{equation*}
\end{proof}

We are going to use the next lemma, together with the inequality established in Proposition \ref{prop: measure-inequality}, to fix a monotone non-increasing representative of $\rho b$.

\begin{lem}\label{lem:MonotoneRep}
Let $\psi\in L^p_{loc}((0,\infty))$, $p\in [1,\infty)$. If the distribution $\psi'$ is nonnegative in the sense that $\langle \psi', \phi\rangle\ge 0$ for any $\phi\in C_c^{\infty}((0,\infty))$ and $\phi\ge 0$, then $\psi$ has a monotone non-decreasing representative.    
\end{lem}

\begin{proof}
    By the Riesz representation theorem, the distribution $\psi'$ is in fact a Radon measure, denoted by $\mathbf{\psi}'$. In particular, we obtain a representative of $\psi$ by defining  $\psi(r)\coloneq \mathbf{\psi}'((0,r])$, for $r>0$.
\end{proof}
Combining $-(\rho b)'\ge 0$, deduced from Proposition \ref{prop: measure-inequality}, and Lemma \ref{lem:MonotoneRep}, we infer that $\rho b$ has a monotone non-increasing representative and  this in turn gives a pointwise representative of $b\in L^2_{loc}(0,\infty)$.

The next observation is about the warping function $f(r)$ in the subquadratic volume growth case. Recall that the density $\rho$ is $\frac1N$-concave and positive on $(0,\infty)$, hence non-decreasing. As observed in the proof of Proposition \ref{prop: measure-inequality}, the distribution $-\rho''\rho+(\rho')^2$ is in fact a Radon measure. So we can test it against characteristic functions, not only smooth functions with compact support. 

\begin{lem}\label{lem:vol<2-to-monotone}
   There exists a sequence $r_i\to\infty$ such that 
   $$-\rho''\rho+(\rho')^2 ([r_i-2,r_i+2])\to 0.$$
   Moreover, $b\ge 0$, so $f$ is non-decreasing on $(0,\infty)$.
\end{lem}

\begin{proof}
We recall that, thanks to Lemma \ref{rem:volumegrowth-to-rho}, we have $\lim_{r\to\infty}\rho(r)/r=0$. We divide the proof into several steps.

\textbf{Step 1.} $\rho'(r)\to 0$ as $r\to\infty$. 
     
Write $u=\rho^{1/N}$, then $u$ is concave. In turn, $u$, hence $\rho$, has well-defined left and right derivatives at every point, and they coincide with at most countably many exceptions. We can take either the left or right derivative as the canonical representation of $u'$ and $\rho'$. The concavity of $u$ gives that for $0<r_1<r_2<r_3$, 
\[
\frac{u(r_2)-u(r_1)}{r_2-r_1}\ge \frac{u(r_3)-u(r_1)}{r_3-r_1}\ge \frac{u(r_3)-u(r_2)}{r_3-r_2}.
\]
Fix any $r>0$, $\epsilon\in (0,r)$. Taking $r_1=r$, $r_2=2r-\epsilon$, $r_3=2r$, we get that 
\[
\frac{\rho^{1/N}(2r)}{r}\ge\frac{u(2r)-u(r)}{r}\ge \frac{u(2r)-u(2r-\epsilon)}{\epsilon}\to u'_{-}(2r)=\frac{1}{N}\rho^{\frac{1}{N}}(2r)\frac{\rho'_{-}(2r)}{\rho(2r)}, \text{ as $\epsilon\to 0^+$}.
\]
Similarly, taking $r_1=r$, $r_2=2r$, $r_3=2r+\epsilon$, we get 
\[
\frac{\rho^{1/N}(2r)}{r}\ge\frac{u(2r)-u(r)}{r}\ge \frac{1}{N}\rho^{\frac{1}{N}}(2r)\frac{\rho'_{+}(2r)}{\rho(2r)}.
\]
Taking either the left or right derivative as the representative of $\rho'$, we have 
\[
\rho'(2r)\le N\frac{\rho(2r)}{r}\to 0, \text{ as $r\to\infty$.}
\]


 \textbf{Step 2.} There exists a sequence $r_i\to\infty$ such that $-\rho''\rho+(\rho')^2 ([r_i-2,r_i+2])\to 0$. 
 
Note that Proposition \ref{prop: measure-inequality} implies $-\rho''\rho+(\rho')^2\ge \rho^2 b^2 $ is a nonnegative measure. If this is not true, then there exist constants $c>0$, $R>2$ such that for any $r>R$, it holds $-\rho''\rho+(\rho')^2([r-2,r+2])\ge c$. Since $\rho'(r)\to 0$, as $r\to\infty$ we can take $R$ so large that 
\[
\int_{r-2}^{r+2} (\rho')^2\le \frac{c}{2},
\]
for any $r>R$. Then by the contradicting assumption,
\[
\rho(-\rho'')([r-2,r+2])\ge \frac{c}{2},
\]
in particular, $-\rho''$ is a nonnegative measure. We may increase $R$ so that $\rho(r)\le \frac{c}{2}r$ holds for any $r>R$, then 
\[
\rho'(r+2)-\rho'(r-2)= (-\rho'')([r-2,r+2])\ge \frac{2}{c(r+2)}\rho(-\rho'')([r-2,r+2])\ge \frac{1}{r+2}.
\]
Now take $r_k=R+4k$ for positive integer $k$, we have $r_{k+1}-2=r_k+2$. Sum over the above inequality taking into account that $\rho'(r)\to 0$ as $r\to \infty$ we get a contradiction 
\[
\rho'(R+4K+2)-\rho'(R+2)\ge \sum_{k=1}^K \frac{1}{R+4k+2}=O(\log K).
\]

\textbf{Step 3.} $b\ge 0$. 

It follows from the monotonicity of (the pointwise version of) $\rho b$, that $\rho b$ changes sign at most once. So when $r$ is large enough, $\rho^2(r) b^2(r)$ is also monotone. We may assume $\rho^2(r) b^2(r)$ is monotone non-decreasing when $r>R+1$, the other case can be proven by considering the interval $[r_i-1,r_i]$ instead of $[r_i,r_i+1]$ below. For a cutoff function $0\le \chi_i(r)\le 1$, $\chi_i(r)=1$ on $[r_i-1,r_i+1]$ and supported on $[r_i-2,r_i+2]$, it holds
\[
0\le\rho^2(r_i) b^2(r_i)\le \int_{r_i}^{r_i+1} \rho^2 b^2 \le \int_{\mathbb{R}}\rho^2 b^2 \chi_i\le\int_{\mathbb{R}}(-\rho''\rho+(\rho')^2)\chi_i\le \int_{r_i-2}^{r_i+2}-\rho''\rho+(\rho')^2\to 0.
\]
Then by monotonicity again $\rho(r) b(r)\to 0$ as $r\to\infty$, in particular $b\ge 0$. 
\end{proof}

Our proof of \ref{thm:hp_rigid_vol<2} will rely on the following consequence of the splitting Theorem for $\RCD(0,N)$ spaces \cite{Gigli14}.

\begin{prop}\label{prop:monotone-to-splitting}
    If the warping function $f$ is non-decreasing on $(0,\infty)$, then the metric space $(Y,y,d)$ splits isometrically as the halfplane $\R\times [0,\infty)$ with Euclidean metric.
\end{prop}

\begin{proof}
     We first claim that for two points $p=(r_0,0)$ and $(r_0,k)=k\cdot p$ in $Y$, where $r_0>0$ and $k\in \Z$, the minimal geodesic $c$ between them is contained in $\{r\le r_0\}$. To verify this, we project $c$ to $c'$ in $Y'=Y/\Z$. Then $c'$ is a loop based at $p'=(r_0,0)\in Y'$ representing $k\in \Z=\pi_1(Y',p')$; moreover, $c'$ has minimal length among all loops representing $k\in \pi_1(Y',p')$. Suppose that $c$ is not entirely contained in $\{r\le r_0\}$, then there is an open interval $I=[a,b]$ such that $c'(a),c'(b)\in \{r=r_0\}\subseteq Y'$ but $c'|_{(a,b)}$ is outside $\{ r\le r_0 \}$. Because the warping function $f$ is non-decreasing, we can shrink the length of $c'$  by replacing $c'|_I$ by a curve in $\{r=r_0\}$ that is path-homotopic to $c'|_I$. This leads to a contradiction and thus verifies the claim. 

     With this claim, we can use a standard argument by group actions to show that the orbit $G y$ is a line in $Y$. In fact, for each $i\in \mathbb{N}$, we draw a minimizing geodesic $c_i$ from $(i^{-1},0)$ to $(i^{-1},i)$. The claim shows that $c_i$ is contained in $\{r\le i^{-1}\}$. We use some group element $g_i$ to move the midpoint $m_i$ back in $\overline{B}_{1/i}(y)$. Then $g_i\cdot c_i$ is a minimizing geodesic in $\{r\le i^{-1}\}$ with its midpoint on $\overline{B}_{1/i}(y)$ and length $\to \infty$ as $i\to\infty$. Hence $g_i\cdot c_i$ subconverges to a line. It is clear that this line coincides with $Gy$ by construction.

     By Gigli's splitting theorem \cite{Gigli13}, $(Y,y,d)$ splits isometrically as $(\R\times Z,(0,z))$ with the line $\R\times \{z\}$ corresponding to $Gy$. Due to the product metric, for any $w\in Z$, the line $L(t)=(t,w)\in \R\times Z$ satisfies $$d(L(t),Gy)=d(z,w)=:s$$
     for all $t\in\R$. Hence $\mathrm{im}(L)$ coincides with $\{r=s\}$, which is a $G$-orbit. It follows that $G\simeq \R$ acts trivially on $Z$ and thus 
     $$[0,\infty)=Y/G=(\R\times Z)/G=Z.$$
     Therefore, $Y$ is isometric to a Euclidean halfplane $\R\times [0,\infty)$.
\end{proof}

Now we complete the proof of Theorem \ref{thm:hp_rigid_vol<2}.

\begin{proof}[Proof of Theorem \ref{thm:hp_rigid_vol<2}]
    By Lemma \ref{rem:volumegrowth-to-rho}, the volume growth assumption implies that $$\lim_{r\to\infty}\dfrac{\rho(r)}{r}=0.$$ 
    So we can apply Lemma \ref{lem:vol<2-to-monotone} to conclude that $f$ is monotone, and we complete the proof by Proposition \ref{prop:monotone-to-splitting}.
\end{proof}

\subsection{Proof of sharp ray rigidity}

In this section
\[
(Y,y,d,\meas)\text{ denotes an } \RCD(0,N) \text{ space satisfying the hypotheses of Theorem \ref{thm:ray_rigid}}.
\]
We prove Theorem \ref{thm:ray_rigid} in this section using the distributional Ricci lower bound.

By Theorem \ref{thm:cont_Riem_metric} and Lemma \ref{lem:measure density}, we have
$$g=\di r^2 + f^2(r)\di \theta^2,\quad \meas= \rho(r)\di r\di \theta$$
on the subset $Y_+=Y-\{y\}=\{r>0\}$ of $Y$. We start with some simple lemmas on the warping function $f$ and the density $\rho$.

\begin{lem}
   The warping function $f$ satisfies 
   $$\lim_{r\to 0} f(r)=0.$$
\end{lem}

\begin{proof}
   We argue by contradiction and suppose that $\liminf_{r\to 0} f(r)=\delta>0$. Then 
   $$d_g((r,0),(r,\pi))\ge \delta\pi$$
   for all $r>0$. On the other hand, by Proposition \ref{prop:d=d_g_on_reg}, we have
   $$d_g((r,0),(r,\pi))=d((r,0),(r,\pi))\le 2r.$$
   Letting $r\to 0$, we obtain the desired contradiction.
\end{proof}

\begin{lem}\label{lem:RhoSuperlinear}
   If 
   \[
   \lim_{r\to 0^+}\frac{\meas(B_r(y))}{r^2}=+\infty,
   \]
   then $\lim_{r\to0^+} \rho(r)/r=+\infty$.
\end{lem}

\begin{proof}
    Since $\rho$ is monotone non-decreasing, we have
    $$2\pi \cdot r\rho(r)\ge \int_0^r \rho(s)\di s \int_0^{2\pi} \di\theta =\meas(B_r(y)).$$
    Hence as $r\to 0^+$,
    $$\dfrac{\rho(r)}{r}\ge \dfrac{\meas(B_r(y))}{2\pi r^2}\to +\infty.$$
\end{proof}

Next, we use distributional Ricci lower bounds to prove that the warping function $f$ is super linear at $r=0$.

\begin{lem}\label{lem:fSuperlinear}
    In an $\RCD(0,N)$ space $(Y, y, d, \meas)$, where the distance $d$ is induced by the metric completion of  $g=\di r^2 + f^2(r)\di \theta^2$ and whose measure is $\meas=\rho(r)\di r\di \theta $, if $\lim_{r\to0^+}\rho(r)/r=\infty$, then $\lim_{r\to0^+}f(r)/r=\infty$.
\end{lem}
\begin{proof}

Let $0<r<1$. Note that $(\ln \rho)''$ is a Radon measure on $(0,\infty)$ by concavity and recall that $b:= (\ln f)'\in L^2_{loc}(\Reg)$. From the distributional Ricci lower bound in Theorem \ref{lem:BERic-bound} we have 
$$0\le b^2 \le -(\ln \rho)''=-\di (\ln \rho)'$$ 
as measures. Here, we take the representative of $(\ln \rho)'$ as the right derivative, which is right continuous. By Cauchy-Schwarz inequality we can estimate $f$ as follows. 
\begin{align*}
  \ln f(1) - \ln f(r) &= \int_r^1 b \,\di s= \int_r^1 \dfrac{1}{\sqrt{s}}\sqrt{s}b \,\di s \\
  &\le \left(\int_r^1 \dfrac{1}{s} \di s\right)^{1/2} \left(\int_r^1 sb^2(s)\, \di s \right)^{1/2}\\
  &\le  (-\ln r)^{1/2}  \left(\int_r^1 -s \,\di  (\ln \rho(s))' \right)^{1/2}.
\end{align*}
For the second integral above, note also that the right derivative $(\ln \rho)'$ is right continuous, we apply the integration by parts formula for Lebesgue--Stieltjes integrals (c.f. \cite{RealAnalysis}*{(21.67) Theorem (iv)}) to infer that  
\begin{equation}\label{eq:IntByParts}
    \int_r^1 -s\, \di  (\ln \rho(s))'=r(\ln \rho)'(r^-) - \ln \rho(r) + C_1\le r(\ln \rho)'(r) - \ln \rho(r) + C_1,
\end{equation}
where $C_1=(\ln \rho)'(1^+) -\ln \rho(1)$ is some constant independent of $r$. Also, $r^-$ denotes the left limit at $r$, and $1^+$ denotes the right limit at $1$. We have used the the monotonicity of $(\ln \rho)'$ to deduce that $(\ln \rho)'(r^-) \le (\ln \rho)'(r)$. By the concavity of $\rho^{\frac{1}{N}}$, it holds
\[
r(\ln \rho)'(r)=\frac{r\rho'(r)}{\rho(r)}\le N,
\]
see the proof of Lemma \ref{lem:vol<2-to-monotone}.


Plugging the above concavity inequality into \eqref{eq:IntByParts}, we infer that   
$$
\ln f(1) -\ln f(r) \le \sqrt{-\ln r}\cdot \sqrt{C_1-\ln \rho(r)}.
$$
We claim that $\lim_{r\to 0^+} f(r)/r = \infty$. Suppose otherwise, then there is a sequence $r_i\to 0^+$ and a constant $L>0$ such that $f(r_i)\le Lr_i$ for all $i$. Then
$$\ln f(r_i) \le \ln L +\ln r_i.$$
Write $C_2=\ln f(1) - \ln L$, then
$$\dfrac{C_2 - \ln r_i}{\sqrt{ -\ln r_i}}\le \sqrt{C_1 - \ln \rho(r_i)}.$$
We square both sides to get
$$ \ln(r_i^{-1}) +2C_2 +\dfrac{C_2^2}{\ln (r_i^{-1})} \le C_1 - \ln \rho(r_i).$$
Hence, there is some constant $C_3$ such that
$$\ln \rho(r_i) - \ln r_i \le C_3,$$
that is, $\rho(r_i)\le e^{C_3} r_i $, which is a contradiction to  $\lim_{r\to0^+}\rho(r)/r=\infty$.
\end{proof}

The last ingredient towards the sharp ray rigidity is that $\lim_{r\to 0^+} f(r)/r = \infty$ forces the geodesic joining a pair of antipodal points to pass through the cone tip. This is merely a statement about the metric. 

\begin{lem}\label{lem: radial-geodesic}
   We consider a warped product $g=\di r^2 + f^{2}(r)\di \theta^2$ on $(0,\infty)\times S^1$, where $f$ is continuous and $\lim_{r\to 0} f(r)=0$. Let $(Y,d)$ be the metric completion of $((0,\infty)\times S^1, d_g)$, which is homeomorphic to a cone over $S^1$. Suppose that 
   $$\lim_{r\to 0^+} \dfrac{f(r)}{r}=\infty.$$
   Then for $r$ sufficiently small, the radial curve is a minimal geodesic from $(r,0)$ to $(r,\pi)$.
\end{lem}

\begin{proof}
   Recall that $y$ is the unique point in $\{r=0\}$. Besides $d$ on $Y$, we shall also consider the Euclidean distance $d_E$, which comes from the warped product $dr^2 + r^2 d\theta^2$. We note that for every $r\ge0$, $B_r(y)$, the metric balls centered at $y$ under $d$ and $d_E$ coincide as subsets in $Y$. Let $\rho>0$ be sufficiently small so that $f(r)\ge r$ for all $r\in (0,2\rho]$. We note that for any pair $(r,0)$ and $(r,\pi)$, where $r\in(0,\rho)$, the minimal $d$-geodesic between them is contained in $B_{2r}(y)$. Then the identity map on $B_\rho(y)$
   $$\mathrm{id}: (B_\rho(y),d|_{B_\rho(y)}) \to (B_\rho(y),d_E)$$
   is $1$-Lipschitz, where $d|_{B_\rho(y)}$ denotes the (extrinsic) distance on $B_\rho(y)$ induced by $d$. Since $d_E((r,0),(r,\pi))=2r$, the $d$-distance between $(r,0)$ and $(r,\pi)$ is at least $2r$. On the other hand, the radial curve from $(r,0)$ to $(r,\pi)$ has $d$-length exactly $2r$. The result follows.
\end{proof}

Finally, we complete the proof of Theorem \ref{thm:ray_rigid}. 

\begin{proof}[Proof of Theorem \ref{thm:ray_rigid}]
Recall that $d$ is the metric completion of $g=\di r^2+ f^2(r)\di \theta^2$ with $\lim_{r\to 0}f(r)=0$. It suffices to show that $f(r)=0$ for all $r\in (0,\infty)$. If not, we deduce from the assumption on $\meas$ and Lemma \ref{lem:RhoSuperlinear} that $\rho(r)/r\to \infty$ as $r\to 0^+$, then from Lemma \ref{lem:fSuperlinear} that $f(r)/r\to \infty$ as $r\to 0^+$.

If $y$, the unique point in $\{r=0\}$, is a singular point, then by Proposition \ref{prop: reg_is_r>0}, $y$ is the only singular point in $Y$. Meanwhile, the concatenation of the radial segments from $(r,0)$ to $y$ then to $(r,\pi)$ is a geodesic by Lemma \ref{lem: radial-geodesic}. The interior of this geodesic contains only $2$-regular points except for the singular point $y$, which contradicts the H\"older continuity of the tangent cones along a geodesic \cite{Deng20}.

If $y$ is a regular point, then it is a $2$-regular point, since by \cite{Kitabeppu-Lakzian_16} it cannot be a $1$-regular point and by \cite{Kit19} it cannot be a $k$-regular point for $k\ge 3$. For any $r_i\to 0$, we consider convergence
\begin{equation}\label{eq:blow-up regular}
   \begin{CD}
    (r_i^{-1} {Y},y,S^1) @>\mathrm{GH}>> (\R^2,0,G) \\
	@VV\pi V @VV \pi V\\
	(r_i^{-1} [0,\infty),0) @>\mathrm{GH}>> ([0,\infty),0).
    \end{CD}
\end{equation}
Because the abelian group $G$ fixes $0$ and the quotient space $\R^2/G$ is isometric to $[0,\infty)$, it is clear that $G=S^1$. For each $i$,  let 
$$p_i=(r_i,0),\quad \theta_i=\pi/3\in S^1,\quad  q_i=(r_i,\pi/3)=\theta_i\cdot p_i.$$
By Theorem \ref{thm:cont_Riem_metric}, for each $i$ we can draw a piecewise smooth curve $c_i:[0,1]\to Y_+$ from $p_i$ to $q_i$ of constant speed such that
$$\mathrm{length}(c_i)\le (1+\epsilon_i)\cdot d(p_i,q_i).$$
Then associated to the convergence (\ref{eq:blow-up regular}), after passing to some subsequences, we have
$$p_i\to p_\infty,\quad \theta_i\to\theta_\infty ,\quad q_i\to q_\infty = \theta_\infty\cdot p_\infty, \quad c_i\to c_\infty,$$
where $c_\infty$ is the segment between $p_\infty$ and $q_\infty$, and
$$d(p_\infty,0)=1=d(q_\infty,0),\quad \theta_\infty^3=\mathrm{id}\in S^1.$$
By planar geometry, $c_\infty$ is contained in the closed annulus $\overline{A}_{1/2}^1 (0)\subseteq \R^2$. As a consequence, $c_i$ is contained in $\overline{A}_{r_i/3}^{2r_i}(y)$ for all $i$ large. Recall that when $r$ is sufficiently small, we have $f(r)\ge 10 r$. Thus we can estimate the length of $c_i$ by
$$(1+\epsilon_i)\cdot d(p_i,q_i)\ge \mathrm{length}(c_i) \ge \min_{r\in[r_i/3,2r_i]} f(r)\cdot \pi/3 \ge \pi r_i.$$
On the other hand,
$$r_i^{-1}d(p_i,q_i)\to d_{\R^2}(p_\infty,q_\infty)\le 2,$$
which is a desired contradiction.

This completes the proof of Theorem \ref{thm:ray_rigid}.
\end{proof}

\section{Structure of fundamental groups}\label{sec:proof}
We prove Theorems \ref{mainthm:vir_abel} and \ref{mainthm:finite} in this section. We start with some properties on the relative volume growth function $\RV(s)$.

\begin{prop}\label{prop:RV_well_defined}
  Let $M$ an open manifold with $\Ric\ge 0$. Then the function $\RV(s)$ is independent of the base point $p\in M$.
\end{prop}

\begin{proof}
Let $q\in M$, and $d\defeq d(p,q)$. For fixed $r\gg d$ and $s\ge 1$, we have
\begin{align*}
    \frac{\vol(B_{rs}(q))}{\vol (B_r(q))}&\le \frac{\vol(B_{rs+d}(p))}{\vol(B_{r-d}(p))}\\
    &=\frac{\vol(B_{rs+d}(p))}{\vol(B_{rs}(p))}\frac{\vol(B_{rs}(p))}{\vol(B_{r}(p))}\frac{\vol(B_{r}(p))}{\vol(B_{r-d}(p))}\\
    &\le \left(1+\frac{d}{rs}\right)^n \frac{\vol(B_{rs}(p))}{\vol(B_{r}(p))} \left(1-\frac{d}{r}\right)^{-n}.
\end{align*}
Let $r\to \infty$ and by symmetry we get that RV is independent of the base point.
\end{proof}

\begin{lem}\label{lem:limit_measure_ray}
   Let $M$ be an open manifold with $\Ric\ge 0$. Suppose that $M$ satisfies
   $$\lim_{s\to\infty} \dfrac{\RV(s)}{s^\beta} =0,$$
   where $\beta>1$. Then for any measured asymptotic cone $(X,x,d,\meas)$ of $M$, the limit renormalized measure $\meas$ satisfies
   $$ \lim_{s\to\infty} \dfrac{\meas(B_s(x))}{s^\beta}=0,\quad  \lim_{s\to 0^+} \dfrac{\meas(B_s(x))}{s^\beta}=\infty.$$
\end{lem}

\begin{proof}
   The first equality is clear. In fact, let $r_i\to\infty$ such that $\meas$ is the limit renormalized measure from the sequence $(r_i^{-1} M,p)$. Then for $s\ge 1$,   $$\meas(B_s(x))=\lim\limits_{i\to\infty} \dfrac{\vol B_{r_is}(p)}{\vol B_{r_i}(p)}\le \limsup_{r\to\infty} \dfrac{\vol B_{rs}(p)}{\vol B_{r}(p)}=\RV(s).$$
   Then the result follows from the condition $\lim_{s\to\infty} s^{-\beta}\RV(s) =0$.

   Next, we prove the second equality. This time, we fix $\lambda \in (0,1)$. Then
\begin{align*}
\meas(B_\lambda(x))&= \lim\limits_{i\to\infty} \dfrac{\vol B_{r_i\lambda}(p)}{\vol B_{r_i}(p)}\ge \liminf_{r\to \infty} \dfrac{\vol B_{r\lambda}(p)}{\vol B_r(p)}\\
& =\dfrac{1}{\limsup\limits_{r \to \infty} \dfrac{\vol B_{r}(p)}{\vol B_{r\lambda}(p)}} =\dfrac{1}{\limsup\limits_{\lambda r\to \infty} \dfrac{\vol B_{\lambda^{-1}\lambda r}(p)}{\vol B_{\lambda r}(p)}}\\
&=\dfrac{1}{\RV(\lambda^{-1})}.
\end{align*}
As $\lambda\to 0^+$, we obtain
$$\lim_{\lambda\to 0^+} \dfrac{\meas(B_\lambda(x))}{\lambda^\beta} \ge \lim_{\lambda\to 0^+} \dfrac{\lambda^{-\beta}}{\RV(\lambda^{-1})}=\infty.$$
\end{proof}

Next, we prove that a relation between relative volume growth and absolute volume growth. This will be needed later in the proof of Theorem \ref{mainthm:finite}.

\begin{lem}\label{lem:rel_vs_abs_vol}
   Let $M$ be an open manifold. Suppose that $M$ satisfies
   $$\liminf_{s\to\infty} \dfrac{\RV(s)}{s^\beta} =L<1,$$
   where $\beta>0$. Then for any $p\in M$,
   $$\lim_{r\to\infty} \dfrac{\vol B_r(p)}{r^\beta}=0.$$
\end{lem}

\begin{proof}
   We choose $\epsilon>0$ small such that $L+2\epsilon<1$. We fix some $s_0>1$ large such that
   $$\limsup_{r\to\infty}\dfrac{\vol B_{rs_0}(p)}{\vol B_r(p)}=\RV(s_0)\le (L+\epsilon) s_0^\beta.$$
   For this fixed $s_0$, there is $r_0=r(s_0)$ large such that
   $$\dfrac{\vol B_{rs_0}(p)}{\vol B_r(p)} \le (L+2\epsilon) s_0^\beta$$
   for all $r\ge r_0$. Denoting $R=rs_0$, where $r\ge r_0$, then 
   \begin{equation}\label{eq:rel_vs_abs_vol}
   \dfrac{\vol B_{R}(p)}{R^\beta} \le (L+2\epsilon) \cdot \dfrac{\vol B_{R/s_0}(p)}{(R/s_0)^\beta}
   \end{equation}
   for all $R\ge r_0s_0$. For any $R\ge r_0s_0$, we choose $k\ge 1$ the smallest integer such that
   $R/s_0^k \le r_0$. Then iterating (\ref{eq:rel_vs_abs_vol}), we obtain
   $$\dfrac{\vol B_{R}(p)}{R^\beta} \le (L+2\epsilon)^k \dfrac{\vol B_{R/s_0^k}(p)}{(R/s_0^k)^\beta}\le (L+2\epsilon)^k \cdot \max_{\rho\in [r_0/s_0,r_0]} \dfrac{\vol B_\rho(p)}{\rho^\beta}.$$
   Letting $R\to\infty$, then $k\to\infty$ and the result follows.
\end{proof}

We also need a simple lemma that transfers the sublinear diameter growth and relative volume growth conditions from $M$ to a finite normal covering space over $M$.

\begin{lem}\label{lem:pass_to_finite_cover}
    Let $M$ be an open manifold with $\Ric\ge 0$ and let $\overline{M}$ be a finite normal covering space over $M$.\\
    (1) If $M$ has sublinear diameter growth, then $\overline{M}$ either splits isometrically as $\R\times K$, where $K$ is a closed manifold, or has sublinear diameter growth.\\
    (2) If $M$ satisfies $\lim_{s\to\infty} \RV(s)/s^\beta=0$, then $\overline{M}$ satisfies $\lim_{s\to\infty} \overline{\RV}(s)/s^\beta=0$, where $\overline{\RV}(s)$ denotes relative volume growth function on $\overline{M}$.
\end{lem}

\begin{proof}
(1) For any sequence $r_i\to\infty$, we consider the convergence
   \begin{equation*}
  \begin{CD}
   (r_i^{-1} \overline{M},\bar{q},\Gamma) @>GH>> (Y,y,\Gamma_\infty) \\
	@VV\pi V @VV \pi V\\
	(r_i^{-1} M,p) @>GH>> ([0,\infty),0)=(Y/F_\infty,\bar{y}),
   \end{CD}
   \end{equation*}
   where $\Gamma$ is the finite covering group and $\Gamma_\infty$ is a finite group fixing $y$. It follows that $Y$ has rectifiable dimension $1$. By \cite{Kitabeppu-Lakzian_16}, $(Y,y)$ is either a line $(\R,0)$ or a ray $([0,\infty),0)$. Together with the connectedness of all asymptotic cones of $\overline{M}$ under pointed Gromov-Hausdorff topology (see, for example, \cite{Pan19}*{Proposition 2.1}), $\overline{M}$ has a unique asymptotic cone. If the unique asymptotic cone of $\overline{M}$ is a line, then by a standard scaling argument, $\overline{M}$ contains a line and thus splits isometrically as $\R\times K$, where $K$ is a closed manifold; see, for example, \cite{Ye_vol<2}*{Proposition 3.3}. If the unique asymptotic cone of $\overline{M}$ is a ray $([0,\infty),0)$, then $\overline{M}$ has sublinear diameter growth.
   
(2) It is clear that we have 
$$\vol B_r(\bar{p})\ge \vol B_r(p)$$
for all $r>0$. We note that
$$B_r(\bar{p})\subseteq \Gamma \cdot (\pi^{-1}(B_r(p))\cap F),$$
where $F\subseteq \overline{M}$ is the Dirichlet domain centered at $\bar{p}$. It follows that
$$\vol B_r(\bar{p}) \le \#\Gamma \cdot \vol B_r(p).$$
Hence
$$\overline{\RV}(s)=\limsup_{r\to\infty} \dfrac{\vol B_{rs}(\bar{p})}{\vol B_r(\bar{p})}\le \# \Gamma \cdot \limsup_{r\to\infty} \dfrac{\vol B_{rs}({p})}{\vol B_r({p})}=\# \Gamma \cdot \RV(s).$$
The result immediately follows.
\end{proof}

Below we assume that $M$ is an open manifold under the assumptions of Theorem \ref{mainthm:vir_abel}. Both Theorems \ref{mainthm:vir_abel} and \ref{mainthm:finite} will follow from an induction theorem describing the asymptotic geometry of successive $\Z$-covering spaces over $M$, similar to the approach in \cite{NPZ}.

Let $\widehat{M}$ be a covering space of $M$ with a finitely generated torsion-free nilpotent covering group $\Lambda$. $\Lambda$ admits a series of normal subgroups:
\begin{equation}\label{eq:series_subgroups}
\{e\}=\Lambda_0 \triangleleft \Lambda_1 \triangleleft ... \triangleleft \Lambda_{k-1} \triangleleft \Lambda_k =\Lambda,
\end{equation}
such that each $\Lambda_{j+1}/\Lambda_j$ is isomorphic to $\mathbb{Z}$. This corresponds to a tower of successive covering spaces:
\begin{equation}\label{eq:tower_covers}
\widehat{M}=\widehat{M}_0\to \widehat{M}_{1} \to ... \to \widehat{M}_{k-1} \to \widehat{M}_k=M,
\end{equation}
where $\widehat{M}_j = \widehat{M}/\Lambda_j$. It follows from the construction that each covering map $\widehat{M}_j \to \widehat{M}_{j+1}$ has its covering group $\Lambda_{j+1}/\Lambda_j \simeq \mathbb{Z}$.

\begin{thm}[Induction Theorem]\label{thm:induction_sub2}
Let $j=0,1,...,k$ and let $r_i\to\infty$ be a sequence. After passing to a convergent subsequence if necessary, we consider the convergence
\begin{equation}\label{eq:induction}
(r_i^{-1}\widehat{M}_{k-j},\hat{p}_{k-j}) \overset{GH}\longrightarrow (Y_{k-j},y_{k-j}).
\end{equation}
for all $j$. Then, the following holds.\\
(1) $(Y_{k-j},y_{k-j})$, as a pointed metric space, is isometric to either a Euclidean space $(\mathbb{R}^{j+1},0)$ or a Euclidean halfspace $(\mathbb{R}^j \times [0,\infty),0)$.\\
(2) When $Y_{k-j}=\mathbb{R}^j\times [0,\infty)$, it has a limit renormalized measure $\meas_{k-j}=\mathcal{L}^j \otimes \underline{\meas}$, where $\mathcal{L}^j$ is the Lebesgue measure on $\R^j$ and $\underline{\meas}$ is a measure on the ray such that
$$\lim_{r\to 0^+} \dfrac{\underline{\meas}([0,r])}{r^2}=+\infty,\quad \lim_{r\to\infty} \dfrac{\underline{\meas}([0,r])}{r^2}=0.$$
(3) When $Y_{k-j}=\mathbb{R}^j\times [0,\infty)$, for every $s>0$, there is a sequence of domains $D_{i}(s)\subseteq \widehat{M}_{k-j}$ such that\\
(3A) $D_{i}(s)\overset{GH}\to [-s,s]^j\times [0,s] \subset Y_{k-j}$ associated to (\ref{eq:induction});\\
(3B) for every $u\in (-s,s)^j \times [0,s)$ and every sequence $u_i\overset{GH}\to u$ associated to (\ref{eq:induction}), it holds that $u_i\in D_i(s)$ for all $i$ large.
\end{thm}

\begin{rem}
Theorem \ref{thm:induction_sub2}(1) is sufficient for later applications to prove Theorems \ref{mainthm:vir_abel} and \ref{mainthm:finite}. Theorem \ref{thm:induction_sub2}(2,3) are needed in the induction steps to prove Theorem \ref{thm:induction_sub2}(1). 
\end{rem}

\begin{proof}
   Most of the proof in \cite{NPZ}*{Theorem 3.23} directly apply here verbatim. The main differences are a slight modification on the computation of limit renormalized measures, and applying the new plane/halfplane rigidity Theorem \ref{thm:rigid_p/h} instead of \cite{NPZ}*{Theorem 1.8} with linear measure growth. Because the proof of \cite{NPZ}*{Theorem 3.23} spans several sections, for readers' convenience, below we give an outline and write down the differences in detail.

   We start with the base case $j=0$. For (1), because $M$ has sublinear diameter growth, it has a unique asymptotic cone $([0,\infty),0)$. (2) follows immediately from Lemma \ref{lem:limit_measure_ray}. For (3), we simply choose $D_i(s)=B_{r_i s}(p)\subseteq M$.

   Next, we assume that the statements hold for $j$ and verify the inductive step $j+1$. For convenience, we write 
   \begin{equation*}
   (\widehat{N},\hat{q},\Gamma):=(\widehat{M}_{k-(j+1)},\hat{p}_{k-(j+1)},\mathbb{Z}),\quad (N,q):=(\widehat{M}_{k-j},\hat{p}_{k-j}).
   \end{equation*}
   For any sequence $r_i\to\infty$, we have convergence after passing to a subsequence
   \begin{equation}\label{eq:induction_diagram}
  \begin{CD}
   (r_i^{-1} \widehat{N},\hat{q},\Gamma\simeq \mathbb{Z}) @>GH>> (Y,y,G) \\
	@VV\pi V @VV \pi V\\
	(r_i^{-1} N,q) @>GH>> (X,x)=(Y/G,\bar{y}),
   \end{CD}
   \end{equation}
   By inductive assumption (1), the metric space $X$ is isometric to a Euclidean space $\R^{j+1}$ or a Euclidean halfspace $\R^j\times [0,\infty)$. Then we can apply the equivariant convergence techniques developed in \cite{NPZ}*{Sections 3.1-3.3, 3.5} to study (\ref{eq:induction_diagram}). We remark that the intermediate results in these sections from \cite{NPZ} do not rely on the linear volume growth but only need non-maximal escape rate of $(\widehat{N},\hat{q},\Gamma)$, which is guaranteed here thanks to Proposition \ref{prop:non_max_escape_rate}. Afterwards, we apply Theorem \ref{thm:orbit_R} and deduce that the orbit $Gy$ is homeomorphic to $\R$, thus $G=\R \times K$, where $K$ is the isotropy subgroup of $G$ fixing $y$.

   When $X$ is isometric to $\R^{j+1}$, we can lift $j+1$ many independent lines to $Y$. Thanks to Gigli's splitting theorem, $(Y,y)$ splits isomorphically as $(\R^{j+1}\times Z,(0,z))$; moreover, $G$ acts trivially on $\R^{j+1}$ and $Z/G$ is a single point. Since the orbit $Gz$ is homeomorphic to $\R$, $Z$ must be a line. It follows that $Y$ is isomorphic to $\R^{j+2}$.

   The nontrivial case is when $X$ is isometric to $\R^j\times [0,\infty)$. We first pass to a subsequence of $r_i$ such that that $X=\R^j\times [0,\infty)$ has a limit renormalized measure $\meas_X=\mathcal{L}^j\otimes \underline{\meas}$ such that
   $$\lim_{s\to 0^+} \dfrac{\underline{\meas}([0,s])}{s^2}=+\infty,\quad \lim_{s\to\infty} \dfrac{\underline{\meas}([0,s])}{s^2}=0,$$
   using inductive assumption (2). Passing to a subsequence again, $Y$ carries a limit renormalized measure $\meas_Y$. Thanks to Gigli's splitting theorem, $Y$ splits isomorphically as
   $$(Y,y,d_Y,\meas_Y)=(\R^j,0,d_E,\mathcal{L}^j)\otimes (Z,z,d_Z,\meas_Z),$$
   where $(Z,d_Z,\meas_Z)$ is an $\RCD(0,n-j)$ space. Moreover, $G$ acts trivially on $\R^j$ and $Z/G$ is isometric to a ray.

   We shall verify that $Z$ satisfies the condition in Theorem \ref{thm:rigid_p/h}; in other words, we need to estimate $\meas_Z(\Omega_s)$ as $s\to 0$ and $s\to\infty$, respectively. 
   Following the notations in \cite{NPZ}*{Section 3.5}, we set\\
   $\bullet$ $\gamma\in\Gamma$, a generator of $\Gamma\simeq \Z$;\\
   $\bullet$ $m_i$, the smallest positive integer $m$ such that $d(\gamma^m \hat{q},\hat{q})\ge r_i$;\\
  $\bullet$ $S(m_i)=\{ \gamma^m \mid m=0,\pm 1,... \pm \lfloor m_i\rfloor \}$;\\
   $\bullet$ $F\subset \widehat{N}$ the Dirichlet domain centered at $\hat{q}$;\\
$\bullet$ $F_i(s)=\overline{F\cap \pi^{-1}(D_i(s))}\subseteq \widehat{N}$, where $D_i(s)\subseteq N$ is provided in inductive assumption (3).\\
By \cite{NPZ}*{Proposition 3.31 and Corollary 3.37}, we have convergence
\begin{equation}\label{eq:vol_conv}
S(m_i)\cdot F_i(s)\overset{GH}\to [-s,s]^j\times \Omega_s,\quad \meas_i(S(m_i)\cdot F_i(s))\to \meas_Y([-s,s]^j\times \Omega_s)=(2s)^j\meas_Z(\Omega_s),
\end{equation}
associated to the top row of (\ref{eq:induction_diagram}), where $\meas_i=\vol(\cdot)/\vol(B_{r_i}(\hat{q}))$ denotes the renormalized measure on $r_i^{-1}\widehat{N}$. We also have measure convergence
$$\dfrac{\vol D_i(s)}{\vol B_{r_i}(q)}\to \meas_X([-s,s]^j\times [0,s])$$
associated to the bottom row of (\ref{eq:induction_diagram}); see \cite{NPZ}*{Remark 3.39}; it follows that
   \begin{equation}\label{eq:vol_D}
\dfrac{\vol D_i(s)}{\vol D_i(1)}=\dfrac{\vol D_i(s)}{\vol B_{r_i}(q)}\dfrac{\vol B_{r_i}(q)}{\vol D_i(1)} \to \dfrac{\meas_X([-s,s]^j\times [0,s])}{\meas_X([-1,1]^j \times [0,1]) } = s^{j}\cdot \dfrac{\underline{\meas}([0,s])}{\underline{\meas}([0,1])}
\end{equation}
To control the renormalized measure $\meas_i(S(m_i)\cdot F_i(s))$, we also need \cite{NPZ}*{Lemma 3.41}, which implies
$$ \vol B_{r_i/C_1}(\hat{q}) \le \vol \left(S(m_i)\cdot F_i(1)\right) \le \vol B_{C_2r_i}(\hat{q})$$
for some constants $C_1,C_2>0$ independent of $i$.
Together with Bishop-Gromov relative volume comparison, we can pass to a subsequence and assume
\begin{equation}\label{eq:vol_cube}
\dfrac{\vol(S(m_i)\cdot F_i(1)) }{\vol B_{r_i}(\hat{q})} \to \theta \in [1/C_1^n,C_2^n].
\end{equation}
Now we combine (\ref{eq:vol_conv}), (\ref{eq:vol_D}), and (\ref{eq:vol_cube}) to obtain
   \begin{align*}
&(2s)^j\cdot \meas_Z(\Omega_s) =\meas_Y([-s,s]^j\times \Omega_s)\\
=&\lim_{i\to\infty}\dfrac{\vol (S(m_i)\cdot F_i(s))}{\vol B_{r_i}(\hat{q})}=\theta\cdot \lim_{i\to\infty} \dfrac{\vol (S(m_i)\cdot F_i(s))}{\vol (S(m_i)\cdot F_i(1))}\\
=&\ \theta \cdot \lim_{i\to\infty} \dfrac{\# S(m_i) \cdot \vol D_i(s)}{\# S(m_i) \cdot \vol D_i(1)}= \theta \cdot s^{j}\cdot \dfrac{\underline{\meas}([0,s])}{\underline{\meas}([0,1])}.
\end{align*}
We set $c=\theta/(2^j \underline{\meas}([0,1]))$, a constant independent of $s$, then
$$\dfrac{\meas_Z(\Omega_s)}{s^2}=c\cdot \dfrac{\underline{\meas}([0,s])}{s^2}.$$
Letting $s\to 0^+$ and $\to \infty$ respectively, we obtain the desired limits
$$\lim_{s\to 0^+}\dfrac{\meas_Z(\Omega_s)}{s^2}=\infty,\quad \lim_{s\to 0^+}\dfrac{\meas_Z(\Omega_s)}{s^2}=0.$$
Therefore, $Z$ satisfies all the conditions in Theorem \ref{thm:rigid_p/h}. As a result, $Z$ is isometric to a Euclidean plane $\R^2$ or a Euclidean halfplane $\R\times [0,\infty)$. This verifies (1) for the inductive step. Then (2) follows from Remark \ref{rem:p/h_measure}. Lastly, for (3), we set $\widehat{D}_i(s)=S(m_i s)\cdot F_i(s)$, then thanks to \cite{NPZ}*{Proposition 3.31}, we have
$$\widehat{D}_i(s)\overset{GH}\to [-s,s]^{j+1}\times [0,s]\subseteq Y = \R^{j+1}\times [0,\infty)$$
associated to the top row of (\ref{eq:induction_diagram}); (3B) follows from \cite{NPZ}*{Proposition 3.36}. 

This completes the proof of the inductive step.
\end{proof}

\begin{cor}\label{cor:asym_Z_cover}
   Let $M$ be an open $n$-manifold under the assumptions of Theorem \ref{mainthm:vir_abel}. Let $\widehat{M}$ be a $\mathbb{Z}$-folding cover of $M$. Then $(\widehat{M},\mathbb{Z})$ has a unique equivariant asymptotic cone isometric to one of\\
  (1) $(\mathbb{R}^2,0,\mathbb{R}\times \mathbb{Z}_2)$,\\
  (2) $(\mathbb{R}\times [0,\infty),0,\mathbb{R} )$. 
\end{cor}

\begin{proof}
   The result follows directly from Theorem \ref{thm:induction_sub2} with $j=k=1$ and Remark \ref{rem:p/h_action}.
\end{proof}

\begin{proof}[Proof of Theorem \ref{mainthm:vir_abel}]
    The proof is almost identical to the arguments in \cite{NPZ}*{Section 3.7}. The only modification is that we use the above Theorem \ref{thm:induction_sub2} instead of \cite{NPZ}*{Theorem 1.6}, then the proof goes through by verbatim.
\end{proof}

For Theorem \ref{mainthm:finite}, we will prove a splitting result under a slightly weaker condition.

\begin{thm}\label{thm:Z_cover_split}
  Let $M$ be an open manifold with $\Ric\ge 0$. Suppose that $M$ has sublinear diameter growth and
   $$\lim_{s\to\infty} \dfrac{\RV(s)}{s^{2}} =0,\quad \liminf_{r\to\infty} \dfrac{\vol B_r(p)}{r^{1+\delta}}=0 $$ 
   for some $\delta\in(0,1)$. Then the following hold.\\
   (1) For any $\Z$-folding covering space $\widehat{M}$ of $M$, $\widehat{M}$ splits isometrically as $\R\times N$.\\
   (2) If, in addition, the Ricci curvature is positive at a point, then $\pi_1(M)$ is finite.
\end{thm}

\begin{proof}
    (1) Thanks to Corollary \ref{cor:asym_Z_cover}, the covering space $(\widehat{M},\hat{p},\Z)$ has a unique equivariant asymptotic as either $(\mathbb{R}^2,0,\mathbb{R}\times \mathbb{Z}_2)$ 
    or 
    $(\mathbb{R}\times [0,\infty),0,\mathbb{R} )$.
    Then we can use the results from \cite{huang2025} to prove the splitting structure on $\widehat{M}$. More specifically, when $\widehat{M}$ has a unique asymptotic cone $\R^2$, then we can apply the same argument as in \cite{huang2025}*{Lemma 4.1} and conclude that $\widehat{M}$ splits isometrically as $\R^2$ with a closed manifold. When $\widehat{M}$ has a unique asymptotic cone $\R\times[0,\infty)$, then \cite{huang2025}*{Proposition 1.8} implies that $\widehat{M}$ splits off an $\R$-factor isometrically.

    (2) We argue by contradiction and suppose that $\pi_1(M)$ is an infinite group. Because $M$ has sublinear diameter growth, $\pi_1(M)$ is finitely generated \cite{Sor00b}. Hence $\pi_1(M)$ contains a torsion-free nilpotent subgroup of finite index \cites{Milnor68,Gromov81}. After further passing to a subgroup of finite index, we can assume this torsion-free nilpotent group, denoted by $\Lambda$, is normal in $\pi_1(M,p)$. We consider the intermediate covering space
   $$(\overline{M},\bar{p})=(\widetilde{M},\tilde{p})/\Lambda,$$
   where $\widetilde{M}$ is the universal cover of $M$. Then $\overline{M}$ is a finite normal cover over $M$. By Lemmas \ref{lem:rel_vs_abs_vol} and \ref{lem:pass_to_finite_cover}, we can assume that $\overline{M}$ satisfies the assumption in Theorem \ref{thm:Z_cover_split} without loss of generality. Since $\pi_1(\overline{M})=\Lambda$ is a torsion-free nilpotent group, we can construct a $\Z$-folding covering space $\widehat{M}$ over $\overline{M}$. It follows from the first part that $\widehat{M}$ splits off a line isometrically, a contradiction to positive Ricci curvature at a point.
\end{proof}

\begin{proof}[Proof of Theorem \ref{mainthm:finite}]
   Theorem \ref{mainthm:finite} follows directly from Theorem \ref{thm:Z_cover_split}(2) and Lemma \ref{lem:rel_vs_abs_vol}.
\end{proof}

\bibliographystyle{plain}
\bibliography{ref.bib}

\end{document}